\documentclass{article}

\usepackage{amsmath}
\usepackage{amsthm}
\usepackage{amssymb}
\usepackage{amsfonts}
\usepackage{mathrsfs}

\usepackage{subeqnarray}
\usepackage{bm}
\usepackage{tikz}
\usepackage{dsfont}
\usepackage{upgreek}
\usepackage{booktabs}
\usepackage{fancyhdr}
\usepackage{multirow}
\usepackage{multicol}
\usepackage{float}
\usepackage{cases}

\usepackage{enumerate}
\usepackage{subfig}

\usepackage[ruled, vlined, linesnumbered, inoutnumbered, algo2e]{algorithm2e}

\usepackage{bbding}
\usepackage{caption}

\usepackage[mathlines]{lineno}

\usepackage[top = 2.5cm, bottom = 2.5cm, left = 3cm, right = 3cm]{geometry}

\usepackage{graphics, graphicx, epsf, epstopdf}
\graphicspath{{./Figures/}}
\ifpdf
\DeclareGraphicsExtensions{.eps,.pdf,.png,.jpg}
\else
\DeclareGraphicsExtensions{.eps}
\fi

\usepackage[breaklinks, colorlinks, linkcolor = black, citecolor = black, urlcolor = black]{hyperref}

\usepackage{xcolor}
\usepackage{array}
\usepackage{placeins}
\usepackage{pgfplots}
\pgfplotsset{compat = 1.18}
\usetikzlibrary{calc, arrows.meta, positioning}

\providecommand{\U}[1]{\protect\rule{.1in}{.1in}}

\newtheorem{theorem}{Theorem}[section]

\newtheorem{lemma}[theorem]{Lemma}

\newtheorem{definition}[theorem]{Definition}

\newtheorem{assumption}{Assumption}

\numberwithin{equation}{section}

\newcommand{\N}{\mathbb{N}}
\newcommand{\R}{\mathbb{R}}

\newcommand{\zz}{^\top}
\newcommand{\dist}{\textnormal{dist}}

\newcommand{\va}{{\mathbf{a}}}

\newcommand{\vc}{{\mathbf{c}}}
\newcommand{\vd}{{\mathbf{d}}}

\newcommand{\vf}{{\mathbf{f}}}

\newcommand{\vs}{{\mathbf{s}}}

\newcommand{\vu}{{\mathbf{u}}}

\newcommand{\vx}{{\mathbf{x}}}
\newcommand{\vy}{{\mathbf{y}}}

\newcommand{\vD}{{\mathbf{D}}}
\newcommand{\vE}{{\mathbf{E}}}

\newcommand{\vG}{{\mathbf{G}}}

\newcommand{\vI}{{\mathbf{I}}}
\newcommand{\vJ}{{\mathbf{J}}}

\newcommand{\vW}{{\mathbf{W}}}
\newcommand{\vX}{{\mathbf{X}}}

\newcommand{\cD}{{\mathcal{D}}}

\newcommand{\cG}{{\mathcal{G}}}

\newcommand{\cI}{{\mathcal{I}}}
\newcommand{\cJ}{{\mathcal{J}}}
\newcommand{\cK}{{\mathcal{K}}}

\newcommand{\cN}{{\mathcal{N}}}

\newcommand{\cP}{{\mathcal{P}}}
\newcommand{\cQ}{{\mathcal{Q}}}

\newcommand{\cS}{{\mathcal{S}}}
\newcommand{\cT}{{\mathcal{T}}}
\newcommand{\cU}{{\mathcal{U}}}
\newcommand{\cV}{{\mathcal{V}}}

\newcommand{\dkh}[1]{\left(#1\right)}
\newcommand{\hkh}[1]{\left\{#1\right\}}
\newcommand{\fkh}[1]{\left[#1\right]}
\newcommand{\jkh}[1]{\left\langle#1\right\rangle}
\newcommand{\norm}[1]{\left\|#1\right\|}
\newcommand{\abs}[1]{\left\lvert #1\right\rvert}

\DeclareMathOperator*{\argmax}{arg\,max} 
\DeclareMathOperator*{\argmin}{arg\,min}

\title{Beyond Finite-Sum Optimization: A Primal-Dual Method for Decentralized Finite-Max Optimization\thanks{This work is supported by National Key R\&D Program of China (2023YFA1009300), National Natural Science Foundation of China (12125108, 12021001, 12288201), and RGC grant JLFS/P-501/24 for the CAS AMSS-PolyU Joint Laboratory in Applied Mathematics, CAS-Croucher Funding Scheme for Joint Laboratories.}}

\author{
	Zhouxing Luo\thanks{Institute of Computational Mathematics and Scientific/Engineering Computing, Academy of Mathematics and Systems Science, Chinese Academy of Sciences, Beijing, China (\href{mailto:luozhouxing18@mails.ucas.ac.cn}{luozhouxing18@mails.ucas.ac.cn}).}
	\and
	Lei Wang\thanks{Department of Applied Mathematics, The Hong Kong Polytechnic University, Hung Hom, Kowloon, Hong Kong, China (\href{mailto:wlkings@lsec.cc.ac.cn}{wlkings@lsec.cc.ac.cn}).}
	\and
	Xin Liu\thanks{Institute of Computational Mathematics and Scientific/Engineering Computing, Academy of Mathematics and Systems Science, Chinese Academy of Sciences, Beijing, China (\href{mailto:liuxin@lsec.cc.ac.cn}{liuxin@lsec.cc.ac.cn}).}
}

\date{}

\begin{document}


\maketitle

\begin{abstract}
Many decentralized systems are judged by their worst-performing agents rather than their average performance. 
We study finite-max optimization over a connected undirected network, where each agent knows only its own smooth, possibly nonconvex objective. 
Quadratic regularization yields a smooth finite-sum representation with a scalar dual variable and locally computable clipped weights. 
We develop an alternating primal-dual gradient tracking (APDGT) method, using neighbor communication without any global or centralized projections. 
Under smoothness and level-boundedness assumptions, a localized Lyapunov analysis with vanishing regularization establishes bounded primal iterates and Clarke stationarity of every accumulation point of their network average. 
The averaged sequence contains an $\varepsilon$-stationary point, measured by near-active gradients, within $O (\varepsilon^{-3})$ iterations. 
Synthetic tests and decentralized logistic regression on Fashion-MNIST illustrate convergence behavior and potential improvements in worst-agent performance. 
\end{abstract}


\section{Introduction}

\label{sec:intro}

Decentralized optimization \cite{chang2020distributed,nedic2018network,wang2024seeking} is commonly formulated through an additive finite-sum objective, in which a network of agents cooperatively minimizes the average of their local costs. 
Such an averaging principle is natural when the aggregate performance is the quantity of primary interest. 
In many distributed systems, however, the overall quality of operation is determined not by the average performance of agents, but by the worst-performing one. 
A single agent with an excessively large loss, delay, energy consumption, or risk may become the bottleneck of the entire system, and improvements in the remaining agents cannot compensate for such a deficiency. 
This observation gives rise to a different optimization principle. 
Rather than balancing local objectives only in an average sense, one seeks a common decision that directly controls the largest local cost.

The worst-agent viewpoint arises in a broad range of practical applications. 
In decentralized learning with heterogeneous data, different agents may represent distinct populations, environments, or data distributions, and minimizing the average empirical risk can conceal poor performance on underrepresented or difficult agents. 
Controlling the largest local loss instead provides a natural mechanism for worst-agent fairness \cite{diana2021minimax,li2023tilted,mohri2019agnostic,sagawa2020distributionally}.
In wireless communication networks, system performance is often limited by the weakest link or the user with the most unfavorable quality-of-service level. 
Then minimax power allocation therefore seeks to reduce the largest local cost and has been studied in distributed cellular networks \cite{srivastava2011distributed}.
A closely related structure appears in edge computing, where synchronous tasks are frequently paced by a straggling worker. 
The whole task cannot be completed until the slowest agent has finished its computation and communication. 
Accordingly, resource allocation and load balancing can be formulated through the minimization of the maximum agent-wise latency, rather than the average one \cite{wang2022gradient}.



Despite their different physical interpretations, these applications share the same underlying mathematical structure. 
A canonical formulation arising from this perspective is the following finite-max model \cite{diana2021minimax,li2020fair,mohri2019agnostic},
\begin{equation}
	\label{opt:fair-max}
	\min_{\vx \in \mathbb{R}^m}
	\left\{
	\Phi (\vx) := \max_{i \in \cN} f_i (\vx)
	\right\},
\end{equation}
where $\cN := \{1, \dotsc, n\}$ is the set of agents and $f_i: \mathbb{R}^m \to \mathbb{R}$ is the local objective function known only to agent $i$. 
We allow the local functions to be smooth but possibly nonconvex, so that $\Phi$ is in general both nonsmooth and nonconvex. 
By explicitly minimizing the largest local function, model~\eqref{opt:fair-max} provides a principled mechanism for improving worst-case performance, protecting disadvantaged agents, and counteracting the imbalance induced by data heterogeneity. 

While objectives of the form \eqref{opt:fair-max} have been extensively studied in centralized optimization and federated learning, their decentralized counterpart remains considerably less understood, especially in the nonconvex setting. 
To shed further light on the underlying structure, we consider the following equivalent reformulation of problem~\eqref{opt:fair-max},
\begin{equation}
	\label{opt:fair-simplex}
	\min_{\vx \in \mathbb{R}^m} \max_{\vy \in \Delta_n}
	\left\{
	\varphi (\vx, \vy) := \sum_{i = 1}^n y_i f_i (\vx)
	\right\},
	\quad
	\mbox{with }
	\Delta_n := \left\{ \vy \in \mathbb{R}_+^n: \mathbf{1}_n^\top \vy = 1 \right\}.
\end{equation}
Here, $\R_+^n$ denotes the nonnegative orthant in $\R^n$, and $\mathbf{1}_n$ is the $n$-dimensional vector of all ones. 
The $i$-th component of $\vy$, denoted by $y_i$, is accessible only to agent $i$. 
These local components are coupled through a global constraint $\mathbf{1}_n^\top \vy = 1$. 
It is clear that $\Phi (\vx) = \max_{\vy \in \Delta_n} \varphi (\vx, \vy)$. 
The minimax model \eqref{opt:fair-simplex} exposes the weights assigned to active agents by introducing a simplex variable. 
In a decentralized network, there is no central coordinator that can aggregate all local functions, identify the worst-performing agent, or maintain global auxiliary variables. 
Each agent has access only to its own function and communicates with a limited set of neighbors through a prescribed graph. 
However, in both \eqref{opt:fair-max} and \eqref{opt:fair-simplex}, the active worst-case component is globally coupled across the entire network. 
For instance, the projection onto the simplex $\Delta_n$ cannot be evaluated under a decentralized setting. 
These features make standard decentralized methods, which are primarily designed for minimizing the average of local functions, inadequate for directly solving \eqref{opt:fair-max} or \eqref{opt:fair-simplex}. 
Developing decentralized algorithms that can faithfully capture the global finite-max structure while relying only on local computation and neighbor-to-neighbor communication is therefore both theoretically challenging and practically relevant.


\subsection{Related works}

\label{subsec:related}

Decentralized optimization has developed into a mature paradigm for solving large-scale learning and signal-processing problems over peer-to-peer networks. 
In this setting, agents are allowed to communicate only with their neighbors and no central coordinator is available. 
The canonical model in this literature is the following consensus finite-sum optimization problem, 
\begin{equation}
	\label{opt:average}
	\min_{\vx \in \R^m} \hspace{1mm} \dfrac{1}{n} \sum_{i = 1}^n f_i (\vx),
\end{equation}
where the global objective is additively decomposable across the network. 
Recent years have witnessed the rapid development of decentralized algorithms for different objective functions and feasible sets, such as gradient descent algorithms \cite{nedic2009distributed,yuan2016convergence}, gradient tracking methods \cite{xu2015augmented,nedic2017achieving,qu2018harnessing,wang2022decentralized}, primal-dual frameworks \cite{shi2015extra,chang2014multi,hajinezhad2019perturbed}, proximal gradient approaches \cite{shi2015proximal,scutari2019distributed,wang2025decentralized}, and so on. 
Interested readers can refer to some recent surveys \cite{nedic2018network,chang2020distributed} and references therein. 
Despite these advances, the structural premise of this line of work remains the finite-sum model. 
Consequently, such methods cannot directly address problems~\eqref{opt:fair-max} and \eqref{opt:fair-simplex}, which require specialized mechanisms.

Another more recent stream studies decentralized saddle-point and minimax problems. 
Most existing algorithms \cite{rogozin2021decentralized,xian2021faster,xu2024decentralized,you2025decentralized} consider problems of the finite-sum form
\begin{equation}
	\label{opt:minimax}
	\min_{\vx} \max_{\vy} \hspace{1mm} \dfrac{1}{n} \sum_{i = 1}^n g_i (\vx, \vy),
\end{equation}
or closely related formulations with locally stored copies of the primal and dual variables. 
These works have clarified how decentralized communication interacts with the instability of minimax dynamics. 
Nevertheless, their problem structure is fundamentally different from the finite maximum considered in \eqref{opt:fair-simplex}. 
The inner variable in the usual decentralized minimax model \eqref{opt:minimax} is a shared game or dual variable, whereas the simplex variable in \eqref{opt:fair-simplex} selects the globally worst agents. 
Its projection, normalization, and active-set identification require information from all nodes. 
Thus, the globally coupled simplex geometry is not amenable to a direct application of existing decentralized minimax algorithms for standard finite-sum models.


In decentralized settings, several early works have been devoted to solving the finite-max model~\eqref{opt:fair-max} under convexity assumptions. 
By introducing an epigraph variable $t \in \R$, Srivastava et al. \cite{srivastava2011distributed} reformulate problem~\eqref{opt:fair-max} into the following nonsmooth penalty model,
\begin{equation}
	\label{opt:penalty}
	\min_{\vx \in \R^m, \, t \in \R} \hspace{1mm} \dfrac{1}{n} \sum_{i = 1}^n \dkh{ t + n \beta \max \hkh{ f_i (\vx) - t, 0 } },
\end{equation}
where $\beta > 1$ is a penalty parameter. 
This model admits an additive finite-sum decomposition across the agents, thereby making the objective amenable to decentralized implementation. 
Based on this structure, they develop a decentralized subgradient method to solve problem~\eqref{opt:penalty}. 
Although the exact penalty reformulation in \eqref{opt:penalty} itself remains valid beyond convexity, their algorithmic analysis fundamentally exploits convex subgradient inequalities and therefore does not directly carry over to nonconvex functions. 
In the nonconvex regime, developing decentralized algorithms for \eqref{opt:penalty} requires explicitly accounting for the nonsmooth geometry induced by the penalties, which becomes a central obstacle in both algorithm design and convergence analysis.

Another line of work circumvents the finite-max structure through a power-sum approximation. 
Hanada et al. \cite{hanada2017simple} pursue this idea for nonnegative and convex functions by considering
\begin{equation}
	\label{opt:qffl}
	\min_{\vx \in \mathbb{R}^m} \hspace{1mm} \frac{1}{q + 1} \sum_{i = 1}^n \fkh{f_i (\vx)}^{q + 1},
\end{equation}
where $q > 0$ is a constant. 
For large $q$, this problem increasingly emphasizes worst-performing agents and can be interpreted as a smooth proxy for the finite-max model~\eqref{opt:fair-max}. 
The same power-weighting principle later appears in the $q$-fair federated learning ($q$FFL) framework \cite{li2020fair}. 
Following the terminology now commonly used in federated learning, we henceforth refer to such power-sum formulations as $q$FFL. 
The objective function of \eqref{opt:qffl} itself remains a smooth finite-sum structure and can therefore be minimized in a fully decentralized manner using standard consensus or gradient tracking techniques. 
However, for any finite $q$, it does not directly solve the exact finite-max model.
In addition, the approximation quality deteriorates with the network size, as we have
\begin{equation*}
	\Phi (\vx) \leq \fkh{ \sum_{i = 1}^n \fkh{f_i (\vx)}^{q + 1} }^{1 / (q + 1)} \leq n^{1 / (q + 1)} \Phi (\vx).
\end{equation*}
Consequently, maintaining a fixed approximation accuracy as $n$ grows generally requires a larger value of $q$. 
As already observed in \cite{hanada2017simple}, increasing $q$ can substantially magnify the scale of local gradients and slow down or destabilize the decentralized iterations. 
They consequently introduce $q$-dependent stepsizes to mitigate this effect. 

The remaining challenge is to solve the nonconvex finite-max problem using local scalar weights and neighbor communication, without a global simplex projection or a nested consensus procedure for every weight update. 
As the regularization vanishes, the smoothness constants grow and interact with consensus errors, tracking errors, and lagged gradient information induced by the alternating updates. 
Controlling these effects while establishing boundedness and stationarity for the original nonsmooth objective is the main analytical task of this paper.

\subsection{Contributions}

\label{subsec:contribution}

This paper presents a systematic investigation of the decentralized finite-max problem~\eqref{opt:fair-max} with smooth but possibly nonconvex local functions. 
Our main contributions are summarized as follows.

We first uncover a decentralized dual representation of the finite-max geometry in problem~\eqref{opt:fair-simplex}. 
Rather than attempting to distribute the simplex variable itself, we introduce a quadratic regularization in the inner maximization and characterize its value function through a dual variable. 
This viewpoint converts the global selection of the worst-performing agents into locally evaluable responses. 

Guided by this representation, we develop an alternating primal-dual gradient tracking (APDGT) algorithm tailored to the finite-max problem~\eqref{opt:fair-max}. 
Each agent maintains a local model, a copy of the dual variable, and an adaptive weight reflecting the relative magnitude of its current loss. 
Two dynamic tracking mechanisms recover, respectively, the simplex constraint and the active-weighted gradient, while consensus steps reconcile the local primal and dual variables. 
The regularization parameter is gradually driven to zero, allowing APDGT to evolve from a smooth weighted surrogate toward the exact nonsmooth maximum. 
Moreover, APDGT automatically identifies the worst-performing agents through its adaptive weights. 
All operations are performed using local information and exchanges with adjacent agents. 
Neither a central coordinator nor a global simplex projection is required. 

We also establish the global convergence of APDGT through a weighted Lyapunov construction, whose coefficients are synchronized with the regularization schedule. 
Our analysis simultaneously controls descent of the averaged primal-dual dynamics, disagreement among agents, and two tracking errors. 
Under standard assumptions, every accumulation point of the generated sequence is shown to be Clarke-stationary for problem~\eqref{opt:fair-max}. 
Moreover, APDGT reaches an $\varepsilon$-stationary point within $O(\varepsilon^{- 3})$ iterations. 
Numerical experiments on synthetic problems and heterogeneous decentralized learning tasks corroborate the theoretical findings and substantiate the practical implication of directly controlling worst-agent performance.

\subsection{Organization}

The remainder of the paper is organized as follows. 
Section~\ref{sec:model} develops the theoretical foundations of the finite-max model, including its regularized dual representation and the associated stationarity measure. 
Section~\ref{sec:algorithm} introduces the proposed APDGT algorithm and its decentralized implementation. 
Section~\ref{sec:convergence} establishes the global convergence and iteration complexity of APDGT. 
Section~\ref{sec:numerics} presents numerical experiments on both synthetic problems and real-world applications to illustrate the fairness benefits of the proposed framework. 
Section~\ref{sec:conclusion} concludes this paper.

\section{Model analysis}

\label{sec:model}

The goal of this paper is to develop a fully decentralized algorithm for solving problem~\eqref{opt:fair-max} by exploiting its equivalent characterization presented in \eqref{opt:fair-simplex}. 
To address the globally coupled constraint in problem~\eqref{opt:fair-simplex}, we add a regularization term in its inner maximization problem and establish the corresponding dual formulation. 
To some extent, this transformation decouples the global constraint imposed by the simplex. 
Moreover, the stationarity violation of the original finite-max problem can also be derived from it.

\subsection{Blanket assumptions}

\label{subsec:assump}

Throughout this paper, all vectors are understood as column vectors unless otherwise specified. 
We use $\jkh{\cdot, \cdot}$ and $\norm{\,\cdot\,}$ for the standard inner product and norm, interpreted as the Euclidean ones for vectors and the Frobenius ones for matrices. 
Before proceeding to the detailed exposition, we first introduce the blanket assumptions.

\begin{assumption}
	\label{asmp:smooth}
	\mbox{}
	
	\begin{enumerate}[(i)]
		
		\item
		Each function $f_i: \R^m \to \R$ is $L$-smooth, namely,
		\begin{equation*}
			\norm{ \nabla f_i (\vx) - \nabla f_i (\vx^\prime)} \leq L \norm{ \vx - \vx^\prime},
		\end{equation*}
		for all $\vx, \vx^\prime \in \R^m$ and $i \in \cN$. 
		
		\item 
		The finite-max objective $\Phi$ is level-bounded, namely, for every $a \in \R$, the sublevel set $\{\vx \in \R^m \mid \Phi (\vx) \leq a\}$ is bounded. 
		
	\end{enumerate}
\end{assumption}

Since $\Phi$ is continuous, Assumption~\ref{asmp:smooth}(ii) implies that $\Phi$ is bounded below and attains its minimum $\Phi^\star := \min_{\vx \in \R^m} \Phi (\vx) > -\infty$. 
For any compact set $\cK \subseteq \R^m$, we define the local gradient bound
\begin{equation}
	\label{eq:local_grad_bound}
	L_{\cK, 0} := \max_{i \in \cN} \sup_{\vx \in \operatorname{conv} (\cK)} \norm{\nabla f_i (\vx)},
\end{equation}
where $\operatorname{conv} (\cK)$ is the convex hull of $\cK$. 
Since each $f_i$ is continuously differentiable, we have $L_{\cK, 0} < \infty$ for any compact set $\cK$.

\subsection{Regularized dual function}

\label{subsec:algo_design}

This subsection constructs a regularized dual function to approximate the finite-max objective $\Phi$ in problem~\eqref{opt:fair-max}, which can be directly evaluated in a decentralized setting. 

Since the inner maximization problem in \eqref{opt:fair-simplex} may admit multiple solutions, we introduce a quadratic regularization term \cite{nesterov2005smooth} into this problem and define
\begin{equation}
	\label{eq:reg_inner_map}
	\Phi_\rho (\vx) := \max_{\vy \in \Delta_n}
	\hkh{
		\sum_{i = 1}^n y_i f_i (\vx) - \frac{\rho}{2} \norm{\vy}^2
	},
\end{equation}
where $\rho > 0$ is a regularization parameter. 
The objective function in \eqref{eq:reg_inner_map} is $\rho$-strongly concave in $\vy$, and hence its maximizer $\vy_\rho^\star (\vx)$ is unique. 
Since $\norm{\vy} \leq 1$ on $\Delta_n$ and each unit vector is feasible, we have
\begin{equation}
	\label{eq:Phi_rho_approx}
	\Phi (\vx) - \frac{\rho}{2}
	\leq \Phi_\rho (\vx)
	\leq \Phi (\vx).
\end{equation}
Thus, the approximation error between $\Phi_\rho$ and $\Phi$ vanishes as $\rho \downarrow 0$.

For each fixed $\vx \in \R^m$, the maximizer of problem~\eqref{eq:reg_inner_map} is determined by the Euclidean projection
\begin{equation}
	\label{eq:proj_form_new}
	\vy_\rho^\star (\vx)
	= \argmin_{\vy \in \Delta_n} \frac{1}{2} \norm{\vy - \frac{1}{\rho} \vf (\vx)}^2,
\end{equation}
where $\vf (\vx) := [f_1 (\vx), \dotsc, f_n (\vx)]\zz \in \R^n$. 
Directly computing the projection in \eqref{eq:proj_form_new} requires global access to all objective values, which is not available under the decentralized setting. 
To address this issue, we derive the dual formulation of \eqref{eq:reg_inner_map}. 
The Lagrangian dual function with respect to the equality constraint $\mathbf{1}_n^\top \vy = 1$ can be written as
\begin{equation}
	\label{eq:dual}
	h_\rho (\nu, \vx)
	= \nu + \max_{\vy \in [0, 1]^n} \hkh{ \sum_{i = 1}^n (f_i (\vx) - \nu) y_i - \frac{\rho}{2} \norm{\vy}^2 },
\end{equation}
where $\nu \in \R$ is the corresponding Lagrangian multiplier. 
Since the regularized inner problem minimizes a concave quadratic function subject to affine constraints and satisfies Slater's condition, strong duality holds, and we have
\begin{equation}
	\label{eq:dual_exact}
	\Phi_\rho (\vx) = \min_{\nu \in \R} h_\rho (\nu, \vx).
\end{equation}
We refer to $h_\rho (\nu, \vx)$ as the regularized dual function, as it involves a regularization term and a dual variable. 
As we shall see, this formulation enables us to decouple the variable $\vy$ subject to the global constraint $\mathbf{1}_n\zz \vy = 1$.

For fixed $\vx \in \R^m$, $\rho > 0$ and $\nu \in \R$, the maximization problem with respect to $\vy$ in \eqref{eq:dual} admits the following closed-form solution,
\begin{equation}
	\label{eq:y_rho_def}
	\vy_\rho (\nu, \vx)
	:= (y_{\rho, 1} (\nu, \vx), \dotsc, y_{\rho, n} (\nu,\vx)), 
	\;\mbox{with}\;
	y_{\rho, i} (\nu, \vx)
	:= \Pi_{[0,1]} \dkh{ \frac{f_i (\vx) - \nu}{\rho} }.
\end{equation}
Moreover, the dual function $h_\rho$ can be expressed as
\begin{equation*}
	h_\rho (\nu, \vx)
	:= \nu + \rho \sum_{i = 1}^n \psi \dkh{ \frac{f_i (\vx) - \nu}{\rho} },
\end{equation*}
where $\psi: \R \to \R$ is a function defined by
\begin{equation*}
	\psi (t)
	:= \max_{0 \leq u \leq 1} \hkh{ t u - \frac{1}{2} u^2 } 
	= \left\{
	\begin{aligned}
		& 0, && \mbox{if } t \leq 0, \\
		& \frac{1}{2} t^2, && \mbox{if } 0 < t < 1, \\
		& t - \frac{1}{2}, && \mbox{if } t \geq 1.
	\end{aligned}
	\right.
\end{equation*}

Let $\nu_\rho^\star (\vx)$ be a global minimizer of the optimization problem in \eqref{eq:dual_exact} for fixed $\vx \in \R^m$ and $\rho > 0$. 
Then we have $\Phi_\rho (\vx) = h_\rho (\nu_\rho^\star (\vx), \vx)$. 
By strong duality, $\vy_\rho^\star (\vx)$ is the global maximizer of the optimization problem in \eqref{eq:dual} with $\nu = \nu_\rho^\star (\vx)$. 
Consequently, it holds that
\begin{equation*}
	\vy_\rho^\star (\vx)
	= \vy_\rho (\nu_\rho^\star (\vx), \vx).
\end{equation*}
We also denote $y_{\rho,i}^\star(\vx) = y_{\rho, i} (\nu_\rho^\star (\vx), \vx)$. 
Since $\psi^\prime (t) = \Pi_{[0,1]} (t)$, we obtain that
\begin{equation*}
	\nabla_\nu h_\rho (\nu,\vx)
	= 1 - \sum_{i = 1}^n \Pi_{[0,1]} \dkh{ \frac{f_i (\vx) - \nu}{\rho} }
	= 1 - \mathbf{1}_n\zz \vy_{\rho} (\nu, \vx).
\end{equation*}
Thus, the simplex feasibility condition is reduced to the one-dimensional equation
\begin{equation*}
	\nabla_\nu h_\rho (\nu, \vx) = 0,
\end{equation*}
which is satisfied at $\nu = \nu_\rho^\star (\vx)$.

The next lemma establishes the local smoothness properties of the regularized dual function and quantifies the corresponding Lipschitz constants, which will be used in the subsequent analysis.

\begin{lemma}[Local smoothness of $h_\rho$]
	\label{lem:h_smooth}
	Suppose that Assumption~\ref{asmp:smooth} holds. 
	Let $\cK\subset\R^m$ be a compact set with $L_{\cK,0}$ defined in \eqref{eq:local_grad_bound}. 
	Then, for any $\rho > 0$, the following statements hold. 
	\begin{enumerate}[(i)]
		
		\item
		For fixed $\vx$, the partial gradient $\nabla_\nu h_\rho (\nu, \vx) = 1 - \mathbf{1}_n\zz \vy_{\rho} (\nu, \vx)$ is $(n / \rho)$-Lipschitz continuous in $\nu$. 
		
		\item
		For fixed $\nu$, the partial gradient $\nabla_{\vx} h_\rho (\nu, \vx) = \sum_{i = 1}^n y_{\rho, i} (\nu, \vx) \nabla f_i (\vx)$ is $(n L + n L_{\cK,0}^2 / \rho)$-Lipschitz continuous in $\vx$ on $\cK$. 
		
		\item
		For fixed $\vx\in\cK$, $\nabla_{\vx} h_\rho (\nu, \vx)$ is $(n L_{\cK,0} / \rho)$-Lipschitz continuous in $\nu$. 
		
		\item
		For fixed $\nu$, $\nabla_\nu h_\rho (\nu, \vx)$ is $(n L_{\cK,0} / \rho)$-Lipschitz continuous in $\vx$ on $\cK$. 
		
	\end{enumerate}
\end{lemma}

\begin{proof}
	According to the definition of $y_{\rho, i} (\nu, \vx)$ in \eqref{eq:y_rho_def}, the nonexpansiveness of $\Pi_{[0,1]}$, and the definition of $L_{\cK,0}$, we have
	\begin{equation}
		\label{eq:local_y_lipschitz}
		\abs{ y_{\rho, i} (\nu, \vx) - y_{\rho, i} (\nu^\prime, \vx^\prime) }
		\leq \rho^{-1}\dkh{ \abs{\nu-\nu^\prime} + L_{\cK,0} \norm{\vx - \vx^\prime} },
	\end{equation}
	for all $\vx,\vx^\prime\in\cK$ and all $\nu,\nu^\prime\in\R$. 
	Here, we have used the fact that the line segment between $\vx$ and $\vx^\prime$ is contained in $\operatorname{conv}(\cK)$. 
	Part~(i) follows by setting
	$\vx=\vx^\prime$ in \eqref{eq:local_y_lipschitz} and summing over $i$. 
	For part~(ii), using $0 \leq y_{\rho, i} (\nu, \vx)\leq 1$, $L$-smoothness of $f_i$, and $\norm{\nabla f_i (\vx)} \leq L_{\cK,0}$ for $\vx \in \cK$ gives
	\begin{equation*}
		\begin{aligned}
			& \norm{ y_{\rho, i} (\nu, \vx) \nabla f_i (\vx) - y_{\rho, i} (\nu, \vx^\prime) \nabla f_i (\vx^\prime) } \\
			\leq {} & \norm{\nabla f_i(\vx)-\nabla f_i(\vx^\prime)} + \abs{y_{\rho,i}(\nu,\vx)-y_{\rho,i}(\nu,\vx^\prime)} \norm{\nabla f_i(\vx^\prime)} \\
			\leq {} & \dkh{ L + \frac{L_{\cK,0}^2}{\rho} } \norm{\vx - \vx^\prime}.
		\end{aligned}
	\end{equation*}
	Summing over $i$ proves part~(ii). 
	Parts~(iii)--(iv) follow from the same bound \eqref{eq:local_y_lipschitz}. 
\end{proof}

\subsection{Stationarity measure}

\label{subsec:moreau_stationarity}

We next define the stationarity measure used for problem~\eqref{opt:fair-max}. 
An interesting observation is that the stationarity violation can be controlled by the gradient of the regularized dual function.

For any $\varepsilon \geq 0$ and $\vx \in \R^m$, we define the $\varepsilon$-near-active index set by
\begin{equation*}
	\cI_\varepsilon (\vx)
	:= \hkh{ i \in \cN \mid \Phi (\vx) - f_i (\vx) \leq \varepsilon },
\end{equation*}
and the associated $\varepsilon$-near-active gradient hull by
\begin{equation*}
	\cG_\varepsilon (\vx)
	:= \operatorname{conv} \hkh{ \nabla f_i (\vx) \mid i \in \cI_\varepsilon (\vx) }.
\end{equation*}
When $\varepsilon = 0$, the set $\cI_0 (\vx)$ is precisely the active index set of $\Phi$ at $\vx$. 
Thus, the Clarke subdifferential of $\Phi$, denoted by $\partial \Phi$, can be expressed as
\begin{equation}
	\label{eq:clarke_max_rule}
	\partial \Phi (\vx)
	= \cG_0 (\vx)
	= \operatorname{conv} \hkh{ \nabla f_i (\vx) \;\middle|\; i \in \argmax_{j \in \cN} f_j (\vx) },
\end{equation}
by the standard calculus rule for the Clarke subdifferential of a finite maximum of continuously differentiable functions~\cite{clarke1990optimization,roc1998var}. 
We use the following notions of stationarity.

\begin{definition}
	\label{def:epsilon_stationary_new}
	We say a point $\vx \in \R^m$ is a Clarke-stationary point of problem~\eqref{opt:fair-max} if $0 \in \partial \Phi (\vx)$.
	For any $\varepsilon > 0$, a point $\vx \in \R^m$ is called an $\varepsilon$-stationary point of problem~\eqref{opt:fair-max} if $\dist \dkh{0, \cG_\varepsilon (\vx) } \leq \varepsilon$. 
\end{definition}


After introducing the stationarity notions, we now connect them to quantities that can be evaluated through the regularized dual function.

\begin{lemma}
	\label{lem:bridge}
	Suppose that Assumption~\ref{asmp:smooth} holds. 
	Let $\cK \subseteq \R^m$ be a compact set with $L_{\cK, 0}$ defined in \eqref{eq:local_grad_bound}. 
	Then, for any $\nu \in \R$, $\vx \in \cK$, and $\rho > 0$,
	\begin{equation*}
		\norm{\nabla \Phi_\rho(\vx)}
		\leq
		\norm{\nabla_{\vx} h_\rho(\nu,\vx)}
		+
		L_{\cK,0} \abs{\nabla_\nu h_\rho(\nu,\vx)}.
	\end{equation*}
\end{lemma}

\begin{proof}
	According to Danskin's theorem~\cite{danskin1966theory}, we have
	\begin{equation*}
		\nabla\Phi_\rho(\vx)
		=
		\sum_{i=1}^n y_{\rho,i}^\star(\vx) \nabla f_i(\vx).
	\end{equation*}
	Recall that $y_{\rho,i}^\star(\vx) = y_{\rho, i} (\nu_\rho^\star (\vx), \vx)$. 
	Each map $\nu \mapsto y_{\rho,i}(\nu,\vx) = \Pi_{[0,1]} ( (f_i(\vx)-\nu) / \rho)$ is nonincreasing for fixed $\vx \in \R^m$. 
	Hence, the differences $y_{\rho,i}(\nu,\vx) - y_{\rho,i}(\nu_\rho^\star(\vx),\vx)$ have the same sign for all $i \in \cN$, which implies that
	\begin{equation*}
		\begin{aligned}
			\norm{ \vy_\rho(\nu,\vx) - \vy_\rho^\star(\vx) }_1
			&=
			\abs{
				\sum_{i=1}^n
				\bigl(
				y_{\rho,i}(\nu,\vx)-y_{\rho,i}(\nu_\rho^\star(\vx),\vx)
				\bigr)
			} \\
			&=
			\abs{
				1-\sum_{i=1}^n y_{\rho,i}(\nu,\vx)
			}
			=
			\abs{ \nabla_\nu h_\rho(\nu,\vx) }.
		\end{aligned}
	\end{equation*}
	Using $\|\nabla f_i(\vx)\|\le L_{\cK,0}$ for $\vx\in\cK$, we obtain
	\begin{equation*}
		\begin{aligned}
			\norm{ \nabla \Phi_\rho (\vx) - \nabla_{\vx} h_\rho (\nu, \vx) }
			&=
			\norm{
				\sum_{i=1}^n
				\bigl(
				y_{\rho,i}^\star(\vx)-y_{\rho,i}(\nu,\vx)
				\bigr)\nabla f_i(\vx)
			} \\
			&\le
			L_{\cK,0}
			\sum_{i = 1}^n
			\abs{
				y_{\rho,i}^\star(\vx)-y_{\rho,i}(\nu,\vx)
			} \\
			&=
			L_{\cK,0} \norm{ \vy_\rho(\nu,\vx) - \vy_\rho^\star(\vx)}_1\\
			&=
			L_{\cK,0} \abs{\nabla_\nu h_\rho(\nu,\vx)}.
		\end{aligned}
	\end{equation*}
	The asserted result then follows from the triangle inequality.
\end{proof}

The next lemma shows that the gradient of $\Phi_\rho$ belongs to the convex hull of $\rho$-near-active gradients, which yields a finite-accuracy stationarity certificate for the original nonsmooth objective function.

\begin{lemma}
	\label{lem:smoothed_gradient_in_hull}
	Suppose that Assumption~\ref{asmp:smooth} holds. 
	Let $\cK\subset\R^m$ be compact and let $L_{\cK,0}$ be defined by
	\eqref{eq:local_grad_bound}. Then, for any $\rho>0$ and $\vx\in\cK$, the support of $\vy_\rho^\star(\vx)$ is contained in the $\rho$-near-active index set, namely,
	\begin{equation*}
		\cJ_\rho (\vx)
		:=
		\hkh{ i \in \cN \mid y_{\rho,i}^\star(\vx)>0 }
		\subseteq
		\cI_\rho(\vx).
	\end{equation*}
	Consequently, we have
	\begin{equation*}
		\nabla \Phi_\rho(\vx) \in \cG_\rho(\vx).
	\end{equation*}
	Moreover, for any $\nu\in\mathbb{R}$, it holds that
	\begin{equation}
		\label{eq:dist_bound_rho_hull}
		\dist \dkh{ 0, \cG_\rho(\vx) }
		\leq
		\norm{ \nabla_{\vx} h_\rho(\nu,\vx) }
		+
		L_{\cK,0} \abs{\nabla_\nu h_\rho(\nu,\vx)}.
	\end{equation}
\end{lemma}

\begin{proof}
	Fix $k\in\argmax_{j \in \cN} f_j(\vx)$. We first note that
	$y_{\rho,k}^\star(\vx)>0$. Otherwise, the clipped representation in \eqref{eq:y_rho_def} gives
	$f_k(\vx)\leq\nu_\rho^\star(\vx)$. Since $f_i(\vx)\leq f_k(\vx)$ for all $i \in \cN$, this
	would imply $y_{\rho,i}^\star(\vx)=0$ for every $i$, contradicting
	$\vy_\rho^\star(\vx)\in\Delta_n$.

	Now we take any $i\in\cJ_\rho(\vx)$. 
	If $i=k$, then $\Phi(\vx)-f_i(\vx)=0$. 
	If $i\ne k$, then $y_{\rho,i}^\star(\vx)>0$ and $y_{\rho,k}^\star(\vx)>0$. 
	Since $\vy_\rho^\star(\vx)\in\Delta_n$, both coordinates lie in $(0, 1)$. 
	Hence the clipping is inactive on these two coordinates, and
	\begin{equation*}
		\Phi(\vx)-f_i(\vx)
		=
		\rho\bigl(
		y_{\rho,k}^\star(\vx)-y_{\rho,i}^\star(\vx)
		\bigr)
		\le \rho.
	\end{equation*}
	Thus, we have $\cJ_\rho(\vx)\subseteq\cI_\rho(\vx)$. 
	It follows from Danskin's theorem~\cite{danskin1966theory} that
	\begin{equation*}
		\nabla \Phi_\rho (\vx)
		= \sum_{i = 1}^n y_{\rho, i}^\star (\vx) \nabla f_i (\vx)
		= \sum_{i \in \cJ_\rho (\vx)} y_{\rho, i}^\star (\vx) \nabla f_i (\vx).
	\end{equation*} 
	Since $\vy_\rho^\star(\vx)\in\Delta_n$ and
	$\cJ_\rho(\vx)\subseteq\cI_\rho(\vx)$, this is a convex combination of
	gradients indexed by $\cI_\rho(\vx)$, implying $\nabla\Phi_\rho(\vx)\in\cG_\rho(\vx)$.
	Therefore, we conclude that
	\begin{equation*}
		\dist \dkh{ 0, \cG_\rho(\vx) }
		\leq
		\norm{\nabla\Phi_\rho(\vx)},
	\end{equation*}
	and \eqref{eq:dist_bound_rho_hull} follows from Lemma~\ref{lem:bridge} applied to $\cK$.
\end{proof}

\section{Algorithm design}

\label{sec:algorithm}

Building upon the regularized dual function, we devise an efficient decentralized algorithm for solving problem~\eqref{opt:fair-simplex}. 
Specifically, each agent maintains the component of the simplex variable assigned to it by updating its local dual variable. 
The simplex violation is tracked by the averaged dual residual, whose convergence is established jointly with consensus and tracking-error control.

\subsection{Network setting}

We consider a network $\cS = (\cN, \cV)$ with a mixing matrix $\vW = [w_{ij}] \in \R^{n \times n}$, where $\cV$ is the set of communication links. 
The neighborhood of agent $i$ is denoted by $\cN_i := \{ j \in \cN \mid w_{ij} > 0\}$.

\begin{assumption}
	\label{asmp:graph}
	The network $\cS$ is undirected and connected. 
	The matrix $\vW$ is nonnegative, symmetric, and doubly stochastic $(\vW \mathbf{1}_n = \vW\zz \mathbf{1}_n = \mathbf{1}_n)$.
	Moreover, $w_{ii}>0$ for every $i \in \cN$, and for $i \neq j$, $w_{ij} > 0$ if and only if $(i, j) \in \cV$.
\end{assumption}

The above assumption about the network and the mixing matrix is standard in the literature \cite{nedic2018network}, under which $\vW$ is primitive, doubly stochastic, and conforms to the underlying network topology. 
Let $\vJ := \mathbf{1}_n \mathbf{1}_n\zz / n$. 
Then invoking the results of \cite[Theorem~1]{xiao2004fast}, we have
\begin{equation*}
	\rho_W
	:= \lambda_{\max} \dkh{ \vW - \vJ}
	< 1,
\end{equation*}
where $\lambda_{\max} (\cdot)$ denotes the largest singular value of a given matrix. 
Moreover, for any $\boldsymbol{\Omega} \in \R^{n \times p}$ with $\bar{\boldsymbol{\omega}} := \mathbf{1}_n\zz \boldsymbol{\Omega} / n \in \R^{1 \times p}$, it holds that
\begin{equation}
	\label{eq:mixing_contraction}
	\norm{ \vW \boldsymbol{\Omega} - \mathbf{1}_n \bar{\boldsymbol{\omega}} }
	= \norm{ (\vW - \vJ) (\boldsymbol{\Omega} - \mathbf{1}_n \bar{\boldsymbol{\omega}}) }
	\leq \rho_W \norm{ \boldsymbol{\Omega} - \mathbf{1}_n \bar{\boldsymbol{\omega}} }.
\end{equation}
This result will be used throughout the consensus-error analysis.

\subsection{Algorithmic framework}

We now derive the decentralized updates from the regularized dual function. 
Motivated by the insights from \eqref{eq:Phi_rho_approx} and \eqref{eq:dual_exact}, problem~\eqref{opt:fair-max} can be approximated by
\begin{equation}
	\label{opt:regularized_dual}
	\min_{\vx \in \R^m, \, \nu \in \R} h_\rho (\nu, \vx).
\end{equation}
This is a smooth minimization problem, which is relatively easy to solve and does not involve any globally coupled constraints. 
As $\rho$ approaches zero, the minimized value $\min_{\nu \in \R} h_\rho (\nu, \vx)$ converges uniformly to $\Phi (\vx)$. 

The structure of problem~\eqref{opt:regularized_dual} is also particularly amenable to decentralized implementation. 
Specifically, the regularized dual function $h_\rho$ admits the following finite-sum decomposition,
\begin{equation}
	\label{eq:h_rho_decomp_new}
	h_\rho (\nu, \vx) = \frac{1}{n}\sum_{i = 1}^n h_{\rho, i} (\nu, \vx),
	\quad\mbox{with}\quad
	h_{\rho, i} (\nu, \vx) := \nu + \rho n \psi \dkh{ \frac{f_i(\vx)-\nu}{\rho} }.
\end{equation}
For each agent $i$, we have
\begin{equation*}
	\nabla_\nu h_{\rho, i} (\nu, \vx) = 1 - n y_{\rho, i} (\nu, \vx),
	\quad\mbox{and}\quad
	\nabla_{\vx} h_{\rho, i} (\nu, \vx) = n y_{\rho, i} (\nu, \vx) \nabla f_i (\vx).
\end{equation*}
Equivalently, the averaged local contributions
$1/n-y_{\rho,i}(\nu,\vx)$ and
$y_{\rho,i}(\nu,\vx)\nabla f_i(\vx)$ sum to the two gradients of
$h_\rho$:
\begin{equation*}
	\nabla_\nu h_\rho(\nu,\vx)=\sum_{i=1}^n\left(\frac1n-y_{\rho,i}(\nu,\vx)\right),
	\qquad
	\nabla_\vx h_\rho(\nu,\vx)=\sum_{i=1}^n y_{\rho,i}(\nu,\vx)\nabla f_i(\vx).
\end{equation*}
Thus, the scalar residual enforces the simplex feasibility condition, while the weighted primal gradient drives the descent in the model variable.

The complete procedure of our algorithm is summarized in Algorithm~\ref{alg:defair_pdc}, which is named {\it Alternating Primal-Dual Gradient Tracking} and abbreviated as APDGT. 
The proposed algorithm leverages gradient tracking to solve problem~\eqref{opt:regularized_dual} and integrates an iterative mechanism for gradually diminishing the regularization parameter. 
Based on dynamic average consensus \cite{zhu2010discrete}, APDGT introduces two tracking variables. The scalar $s_{i,k}$ tracks the averaged dual residual contribution $1/n-y_{i,k}$, and $\vd_{i,k}$ tracks the averaged weighted-gradient contribution $y_{i,k}\nabla f_i(\vx_{i,k-1})$. 
Each agent then updates $\nu_{i,k}$ and $\vx_{i,k}$ along these tracked directions, while updating its clipped simplex weight locally. 
Compared with gradient tracking methods designed for the finite-sum model~\eqref{opt:average}, APDGT incurs only the additional communication of two scalar variables $\nu$ and $s$. 
Therefore, the resulting increase in communication burden is negligible.

\begin{algorithm2e}[t]
	\caption{Alternating Primal-Dual Gradient Tracking (APDGT).}
	\label{alg:defair_pdc}
	
	\KwIn{Initial points $\vx_{\mathrm{init}}$ and $\nu_{\mathrm{init}}$, mixing matrix $\vW$, stepsizes $\{\eta_{\nu, k}\}_{k \geq 0}$ and $\{\eta_{x, k}\}_{k \geq 0}$, and regularization schedule $\{\rho_k\}_{k \geq 0}$.}
	
	\For{agents $i \in \cN$ in parallel}{
		
		Set $\vx_{i, 0} = \vx_{\mathrm{init}}$ and $\nu_{i, 0} = \nu_{\mathrm{init}}$.
		
		Set $y_{i, 0} = \Pi_{[0, 1]} ((f_i (\vx_{i, 0}) - \nu_{i, 0}) / \rho_0)$ and $s_{i, 0} = 1 / n - y_{i, 0}$.
		
		Set $\vx_{i, -1} = \vx_{i, 0}$ and $\vd_{i, 0} = y_{i, 0} \nabla f_i (\vx_{i, 0})$. 
		
	}
	
	\For{$k = 0, 1, \dotsc, K - 1$}{
		
		\For{agents $i \in \cN$ in parallel}{
			
			Communicate with neighbors and update the dual variable by
			\begin{equation*}
				\nu_{i, k + 1} = \sum_{j \in \cN_i} w_{i j} \nu_{j, k} - \eta_{\nu, k} s_{i, k}.
			\end{equation*}
			
			Update the simplex weight by $y_{i, k + 1} = \Pi_{[0,1]} ( (f_i (\vx_{i, k}) - \nu_{i, k + 1}) / \rho_k)$. 
			
			Communicate with neighbors and update the tracking variables by
			\begin{equation*}
				s_{i, k + 1} = \sum_{j \in \cN_i} w_{i j} s_{j, k} + (y_{i, k} - y_{i, k + 1}),
			\end{equation*}
			\begin{equation*}
				\vd_{i, k + 1} = \sum_{j \in \cN_i} w_{i j} \vd_{j, k} + y_{i, k + 1} \nabla f_i (\vx_{i, k}) - y_{i, k} \nabla f_i (\vx_{i, k - 1}).
			\end{equation*}
			
			Communicate with neighbors and update the primal variable by
			\begin{equation*}
				\vx_{i, k + 1} = \sum_{j \in \cN_i} w_{i j} \vx_{j, k} - \eta_{x, k} \vd_{i, k + 1}.
			\end{equation*}
			
		}
		
	}
	
	\KwOut{$\{(\vx_{i, K}, y_{i, K})\}_{i \in \cN}$.}
\end{algorithm2e}

\subsection{Compact form}

\label{subsec:avg_error_metrics}

We proceed to rewrite APDGT in a compact form that isolates the averaged dynamics from the perturbations caused by decentralized communication. 

At each iteration $k \geq 0$, we define
\begin{equation*}
	\begin{aligned}
		\vX_k & := [\vx_{1, k}, \, \dotsc, \, \vx_{n, k}]\zz,
		& \vD_k & := [\vd_{1, k}, \, \dotsc, \, \vd_{n, k}]\zz, \\
		\boldsymbol{\nu}_k & := [\nu_{1, k}, \, \dotsc, \, \nu_{n, k}]\zz,
		& \vs_k & := [s_{1, k}, \, \dotsc, \, s_{n, k}]\zz, \\
		\vy_k & := [y_{1, k}, \, \dotsc, \, y_{n, k}]\zz, 
		& \vG_k & := [y_{1, k} \nabla f_1 (\vx_{1, k - 1}), \, \dotsc, \, y_{n, k} \nabla f_n (\vx_{n, k - 1})]\zz.
	\end{aligned}
\end{equation*}
It is clear that $\vG_0=\vD_0$. 
Then the updates of APDGT can be expressed in the following compact form,
\begin{subequations}
	\label{eq:compact_updates}
	\begin{align}
		\label{eq:compact_nu}
		\boldsymbol{\nu}_{k + 1} & = \vW \boldsymbol{\nu}_k - \eta_{\nu,k} \vs_k, \\
		\label{eq:compact_y}
		\vy_{k+1} & = \Pi_{[0,1]^n} \dkh{ \dkh{ [f_1 (\vx_{1, k}), \dotsc, f_n (\vx_{n, k})]\zz - \boldsymbol{\nu}_{k + 1} } / \rho_k }, \\
		\label{eq:compact_s}
		\vs_{k + 1} & = \vW \vs_k + (\vy_k - \vy_{k + 1}), \\
		\label{eq:compact_d}
		\vD_{k + 1} & = \vW \vD_k + (\vG_{k + 1} - \vG_k), \\
		\label{eq:compact_x}
		\vX_{k + 1} & = \vW \vX_k - \eta_{x, k} \vD_{k + 1}.
	\end{align}
\end{subequations}

We denote the corresponding network averages by
\begin{equation*}
	\bar{\vx}_k := \dfrac{1}{n} \vX_k\zz \mathbf{1}_n, 
	\quad
	\bar{\vd}_k := \dfrac{1}{n} \vD_k\zz \mathbf{1}_n,
	\quad
	\bar{\nu}_k := \dfrac{1}{n} \mathbf{1}_n\zz \boldsymbol{\nu}_k, 
	\quad
	\bar{s}_k := \dfrac{1}{n} \mathbf{1}_n\zz \vs_k.
\end{equation*}
The deviations from these averages are measured by
\begin{equation}
	\label{eq:error_matrix_def}
	\begin{aligned}
		\vE_{x, k} & := \vX_k - \mathbf{1}_n \bar{\vx}_k\zz,
		\quad
		& \vE_{\nu, k} & := \boldsymbol{\nu}_k - \mathbf{1}_n \bar{\nu}_k, \\
		\vE_{d, k} & := \vD_k - \mathbf{1}_n \bar{\vd}_k\zz,
		\quad
		& \vE_{s, k} & := \vs_k - \mathbf{1}_n \bar{s}_k.
	\end{aligned}
\end{equation}
Here, $\vE_{x, k}$ and $\vE_{\nu, k}$ quantify primal and dual disagreement across the network, while $\vE_{d, k}$ and $\vE_{s, k}$ quantify the tracking errors for the primal direction and the dual residual.

Under Assumption~\ref{asmp:graph}, the mixing matrix $\vW$ is doubly stochastic. 
Averaging \eqref{eq:compact_nu} and \eqref{eq:compact_x} therefore gives
\begin{equation}
	\label{eq:average_updates}
	\bar{\nu}_{k + 1} = \bar{\nu}_k - \eta_{\nu, k} \bar{s}_k,
	\quad
	\bar{\vx}_{k + 1} = \bar{\vx}_k - \eta_{x, k} \bar{\vd}_{k + 1}.
\end{equation}
Moreover, the initialization $s_{i, 0} = 1 / n - y_{i, 0}$ and the telescoping structure of the tracking recursions imply that
\begin{equation}
	\label{eq:avg_s_identity}
	n \bar{s}_k = 1 - \mathbf{1}_n\zz \vy_k, 
	\quad
	n \bar{\vd}_{k + 1} = \vG_{k + 1}\zz \mathbf{1}_n = \sum_{i = 1}^n y_{i, k + 1} \nabla f_i (\vx_{i, k}),
\end{equation}
for all $k \geq 0$. 
Consequently, $\bar{s}_k$ is precisely the averaged violation of the simplex constraint, whereas $\bar{\vd}_{k + 1}$ is the averaged weighted-gradient direction.

In the absence of network disagreement and tracking error, the averaged dynamics \eqref{eq:average_updates} coincides with an alternating gradient step on the regularized dual function $h_{\rho_k}$. 
The convergence analysis in the next section quantifies how the errors in \eqref{eq:error_matrix_def} perturb this ideal centralized dynamics.

\section{Convergence analysis}

\label{sec:convergence}

In this section, we rigorously demonstrate the global convergence of APDGT and establish its iteration complexity. 
Our analysis separates consensus and tracking errors, constructs inexact descent for the regularized dual function, and closes the estimates through a weighted Lyapunov argument.

Let $\rho_0 > 0$ be an arbitrary constant. 
Algorithm~\ref{alg:defair_pdc} is equipped with the regularization and stepsize schedules
\begin{equation}
	\label{eq:gc_rho_schedule}
	\rho_k
	=
	\rho_0\left(\frac{\tau}{k+\tau}\right)^{1/3},
	\qquad
	\eta_{\nu,k}=\eta_{x,k}=\alpha\rho_k,
\end{equation}
where $\alpha > 0$ and $\tau \geq 1$ will be chosen later.  
We next fix the compact region on which all local Lipschitz estimates are evaluated. 
Let
\begin{equation*}
	\varsigma := h_{\rho_0} (\bar{\nu}_0, \bar{\vx}_0) + (n + 1) \rho_0 / 2 + 2.
\end{equation*}
We define
\begin{equation}
	\label{eq:localization_sets_fixed}
	\begin{aligned}
		\cQ & := \hkh{ \vx \in \R^m \mid \Phi (\vx) \leq \varsigma }, \\
		\cU & := \operatorname{conv} \hkh{ \vx \in \R^m \mid \dist (\vx, \cQ) \leq 1 }, \\
		\cK & := \hkh{ \vx \in \R^m \mid \dist (\vx, \cU) \leq 1 }.
	\end{aligned}
\end{equation}
By Assumption~\ref{asmp:smooth}(ii), $\cQ$ is compact. 
Hence, both $\cU$ and $\cK$ are compact and convex. 
From now on, $\cK$ is fixed as in \eqref{eq:localization_sets_fixed}, and we write
\begin{equation*}
	\bar{L}_0 := L_{\cK, 0}.
\end{equation*}
If $\bar{L}_0 = 0$, reciprocal bounds involving $\bar{L}_0$ are understood as $+ \infty$. 
For an integer $q \geq 0$, we define the localization event
\begin{equation*}
	\mathscr{H}_{\cK} (q)
	:=
	\hkh{
		\bar{\vx}_k \in \cK, \; \vx_{i, k} \in \cK,
		\;\mbox{for all } i \in \cN \mbox{ and } 0 \leq k \leq q
	}.
\end{equation*}
In the subsequent analysis, all estimates before Subsection~\ref{subsec:localization} are conditional on the appropriate event $\mathscr{H}_{\cK} (q)$. 
Subsection~\ref{subsec:localization} then verifies that these events hold for the whole generated sequence. 
Throughout this section, we adopt the conventions $\rho_{- 1} := \rho_0$, $\eta_{x, - 1} := \eta_{x, 0}$, $\vX_{- 1} := \vX_0$, and $\vE_{x, - 1} := \vE_{x, 0}$.

\subsection{Network errors and gradient mismatch}

\label{subsec:error_analysis}

In this subsection, we record several preparatory results needed for the descent analysis. 

The following lemma first presents the contraction bounds for the network errors defined in \eqref{eq:error_matrix_def}.

\begin{lemma}
	\label{lem:error_contraction}
	Suppose that Assumption~\ref{asmp:graph} holds. 
	Then, for every $k \geq 0$, we have
	\begin{subequations}
		\begin{align}
			\|\vE_{\nu,k+1}\|
			&\le \rho_W\|\vE_{\nu,k}\|+\eta_{\nu,k}\|\vE_{s,k}\|,
			\label{eq:err_nu_bound}\\
			\|\vE_{s,k+1}\|
			&\le \rho_W\|\vE_{s,k}\|+\|\vy_{k+1}-\vy_k\|,
			\label{eq:err_s_bound}\\
			\|\vE_{x,k+1}\|
			&\le \rho_W\|\vE_{x,k}\|+\eta_{x,k}\|\vE_{d,k+1}\|,
			\label{eq:err_x_bound}\\
			\|\vE_{d,k+1}\|
			&\le \rho_W\|\vE_{d,k}\|+\|\vG_{k+1}-\vG_k\|.
			\label{eq:err_d_bound}
		\end{align}
	\end{subequations}
\end{lemma}

\begin{proof}
	From \eqref{eq:compact_nu} and \eqref{eq:average_updates},
	\begin{equation*}
		\vE_{\nu,k+1}
		=
		\boldsymbol{\nu}_{k+1}-\mathbf{1}_n\bar{\nu}_{k+1}
		=
		\vW\vE_{\nu,k}-\eta_{\nu,k}\vE_{s,k}.
	\end{equation*}
	Taking norms and using the inequality~\eqref{eq:mixing_contraction} gives
	\eqref{eq:err_nu_bound}.
	
	Averaging \eqref{eq:compact_s} gives
	$\bar{s}_{k + 1} = \bar{s}_k + \mathbf{1}_n\zz (\vy_k - \vy_{k + 1}) / n$, and hence
	\begin{equation*}
		\vE_{s,k+1}
		=
		\vW \vE_{s,k}
		+
		\fkh{
			(\vy_k-\vy_{k+1})
			-
			\frac{1}{n} \mathbf{1}_n \mathbf{1}_n^\top (\vy_k - \vy_{k + 1})
		}.
	\end{equation*}
	Thus \eqref{eq:err_s_bound} follows from \eqref{eq:mixing_contraction} and $\|\vu - \mathbf{1}_n \mathbf{1}_n\zz \vu / n\| \leq \|\vu\|$ for $\vu \in \R^n$. 
	The estimates \eqref{eq:err_x_bound} and \eqref{eq:err_d_bound} follow analogously from \eqref{eq:compact_x} and \eqref{eq:compact_d}.
\end{proof}

The next estimates establish the discrepancy between the averaged directions in APDGT and the gradients of the regularized dual function.

\begin{lemma}
	\label{lem:mismatch}
	Suppose that Assumption~\ref{asmp:smooth} holds. 
	Then, for every $k \geq 0$ such that $\mathscr{H}_{\cK} (k)$ holds, we have
	\begin{equation}
		\label{eq:mismatch_dual_new}
		\abs{\nabla_\nu h_{\rho_k} (\bar{\nu}_{k + 1}, \bar{\vx}_k) - n \bar{s}_{k + 1}}
		\leq
		\frac{\sqrt{n}}{\rho_k} \norm{\vE_{\nu, k + 1}}
		+
		\frac{\bar{L}_0 \sqrt{n}}{\rho_k} \norm{\vE_{x, k}},
	\end{equation}
	and
	\begin{equation}
		\label{eq:mismatch_primal_new}
		\norm{\nabla_{\vx} h_{\rho_k} (\bar{\nu}_{k + 1}, \bar{\vx}_k) - n \bar{\vd}_{k + 1}}
		\leq
		\frac{\bar{L}_0 \sqrt{n}}{\rho_k} \norm{\vE_{\nu, k + 1}}
		+
		\sqrt{n} \dkh{ L + \frac{\bar{L}_0^2}{\rho_k} } \norm{\vE_{x,k}}.
	\end{equation}
\end{lemma}

\begin{proof}
	Let
	\begin{equation}
		\label{eq:def_hat_y}
		\hat{\vy}_{k+1}
		:=
		\vy_{\rho_k}(\bar{\nu}_{k+1},\bar{\vx}_k)
		=
		\fkh{
			\Pi_{[0,1]}\!\left(\frac{f_1(\bar{\vx}_k)-\bar{\nu}_{k+1}}{\rho_k}\right),
			\ldots,
			\Pi_{[0,1]}\!\left(\frac{f_n(\bar{\vx}_k)-\bar{\nu}_{k+1}}{\rho_k}\right)
		}^\top.
	\end{equation}
	Using $n\bar{s}_{k+1}=1-\mathbf{1}_n^\top\vy_{k+1}$ in \eqref{eq:avg_s_identity} and $\nabla_\nu h_{\rho_k}(\bar{\nu}_{k+1},\bar{\vx}_k) = 1-\mathbf{1}_n^\top\hat{\vy}_{k+1}$ from \eqref{eq:def_hat_y}, we have
	\begin{equation*}
		\bigl|\nabla_\nu h_{\rho_k}(\bar{\nu}_{k+1},\bar{\vx}_k)-n\bar{s}_{k+1}\bigr|
		=
		\bigl|\mathbf{1}_n^\top(\vy_{k+1}-\hat{\vy}_{k+1})\bigr|
		\le
		\sqrt n\,\|\vy_{k+1}-\hat{\vy}_{k+1}\|.
	\end{equation*}
	The nonexpansiveness of $\Pi_{[0,1]}$ and the $\bar L_0$-Lipschitz continuity of $f_i$ yield \eqref{eq:mismatch_dual_new}.

	For the primal mismatch, using \eqref{eq:avg_s_identity}, we obtain
	\begin{equation*}
		\begin{aligned}
			& \nabla_{\vx} h_{\rho_k}(\bar{\nu}_{k+1},\bar{\vx}_k)-n\bar{\vd}_{k+1} \\
			= {} &
			\sum_{i=1}^n
			\hat{y}_{i,k+1}\bigl(\nabla f_i(\bar{\vx}_k)-\nabla f_i(\vx_{i,k})\bigr)
			+
			\sum_{i=1}^n
			(\hat{y}_{i,k+1}-y_{i,k+1})\nabla f_i(\vx_{i,k}).
		\end{aligned}
	\end{equation*}
	Since $\hat{y}_{i,k+1}\in[0,1]$, $\|\nabla f_i\|\le \bar L_0$, and $f_i$ is
	$L$-smooth,
	\begin{equation*}
		\bigl\|\nabla_{\vx} h_{\rho_k}(\bar{\nu}_{k+1},\bar{\vx}_k)-n\bar{\vd}_{k+1}\bigr\|
		\le
		L\sum_{i=1}^n\|\vx_{i,k}-\bar{\vx}_k\|
		+
		\bar L_0\sum_{i=1}^n|\hat{y}_{i,k+1}-y_{i,k+1}|.
	\end{equation*}
	Combining this estimate with the preceding bound on
	$\|\vy_{k+1}-\hat{\vy}_{k+1}\|$ gives \eqref{eq:mismatch_primal_new}.
\end{proof}

In the following lemma, we provide upper bounds on the distance between the iterates generated in two successive steps of our algorithm.

\begin{lemma}
	\label{lem:elementary_variations}
	Let Assumptions~\ref{asmp:smooth} and~\ref{asmp:graph} hold. 
	Then the following two relationships hold for every $k \geq 0$,
	\begin{equation}
		\label{eq:X_increment_sync}
		\norm{\vX_{k+1}-\vX_k}
		\leq
		(1+\rho_W)\norm{\vE_{x,k}}
		+\eta_{x,k}\norm{\vE_{d,k+1}}
		+\sqrt n\,\eta_{x,k}\norm{\bar{\vd}_{k+1}},
	\end{equation}
	and
	\begin{equation}
		\label{eq:nu_increment_sync}
		\norm{\boldsymbol{\nu}_{k+1}-\boldsymbol{\nu}_k}
		\leq
		(1+\rho_W)\norm{\vE_{\nu,k}}
		+\eta_{\nu,k}\norm{\vE_{s,k}}
		+\sqrt n\,\eta_{\nu,k}\abs{\bar s_k}.
	\end{equation}
	Moreover, if $\mathscr{H}_{\cK} (k)$ holds, we have
	\begin{equation}
		\label{eq:y_variation_sync}
		\norm{\vy_{k+1}-\vy_k}
		\leq
		\frac{\bar L_0}{\rho_k}\norm{\vX_k-\vX_{k-1}}
		+\frac{1}{\rho_k}\norm{\boldsymbol{\nu}_{k+1}-\boldsymbol{\nu}_k}
		+\sqrt n\,\delta_k,
	\end{equation}
	and
	\begin{equation}
		\label{eq:G_variation_sync}
		\norm{\vG_{k+1}-\vG_k}
		\leq
		\bar L_0\norm{\vy_{k+1}-\vy_k}
		+L\norm{\vX_k-\vX_{k-1}},
	\end{equation}
	where $\delta_0 := 0$ and $\delta_k := (\rho_{k - 1} - \rho_k) / \rho_{k - 1} = 1 - ((k - 1 + \tau) / (k + \tau))^{1/3}$ for $k \geq 1$.
\end{lemma}

\begin{proof}
	From \eqref{eq:compact_x}, we have
	\begin{equation*}
		\vX_{k + 1} - \vX_k = (\vW - \vI_n) \vX_k - \eta_{x, k} \vD_{k + 1}.
	\end{equation*}
	Using $\vX_k = \vE_{x, k} + \mathbf{1}_n \bar{\vx}_k^\top$, $\vD_{k + 1} = \vE_{d, k + 1} + \mathbf{1}_n \bar{\vd}_{k + 1}^\top$, and $(\vW - \vI_n) \mathbf{1}_n = 0$ gives
	\begin{equation*}
		\vX_{k + 1} - \vX_k = (\vW - \vI_n) \vE_{x, k} - \eta_{x, k} \vE_{d, k + 1} - \eta_{x, k} \mathbf{1}_n \bar{\vd}_{k + 1}^\top.
	\end{equation*}
	Taking Frobenius norms yields \eqref{eq:X_increment_sync}. 
	The proof of \eqref{eq:nu_increment_sync} follows analogously from \eqref{eq:compact_nu}.

	Now we fix $k \geq 0$ such that $\mathscr{H}_{\cK} (k)$ holds. 
	Let
	\begin{equation}
		\label{eq:tilde-y}
		\tilde{\vy}_k
		:=
		\Pi_{[0,1]^n}
		\left(
		\frac{1}{\rho_k}
		\left(
		[f_1(\vx_{1,k-1}),\ldots,f_n(\vx_{n,k-1})]^\top
		-\boldsymbol{\nu}_k
		\right)
		\right).
	\end{equation}
	By the nonexpansiveness of $\Pi_{[0,1]^n}$ and the $\bar{L}_0$-Lipschitz continuity of $f_i$, we have
	\begin{equation}
		\label{eq:dist-yk+1}
		\norm{\vy_{k + 1} - \tilde{\vy}_k}
		\leq
		\frac{\bar{L}_0}{\rho_k} \norm{\vX_k - \vX_{k - 1}}
		+\frac{1}{\rho_k} \norm{\boldsymbol{\nu}_{k + 1} - \boldsymbol{\nu}_k}.
	\end{equation}
	For $k = 0$, the conventions $\vX_{-1} = \vX_0$ and $\rho_{-1} = \rho_0$, together with the initialization of $\vy_0$, give $\tilde{\vy}_0 = \vy_0$. 
	For $k \geq 1$, the monotonicity of $\{\rho_k\}$ and the coordinatewise nonexpansiveness imply
	\begin{equation*}
		\abs{\tilde y_{i, k} - y_{i, k}}
		\leq
		\frac{\rho_{k - 1} - \rho_k}{\rho_{k - 1}}
		=\delta_k,
	\end{equation*}
	for all $i \in \cN$. 
	Consequently, for every $k \geq 0$, we can obtain that
	\begin{equation}
		\label{eq:dist-tilde-y}
		\norm{\tilde{\vy}_k - \vy_k}
		\leq
		\sqrt{n} \delta_k.
	\end{equation}
	Combining two relationships \eqref{eq:dist-yk+1} and \eqref{eq:dist-tilde-y} gives \eqref{eq:y_variation_sync}. 
	Finally, for each $i \in \cN$, it holds that
	\begin{equation*}
		\norm{ [\vG_{k + 1} - \vG_k]_i }
		\leq
		\abs{y_{i, k + 1} - y_{i, k}} \norm{\nabla f_i (\vx_{i, k})}
		+ \abs{y_{i, k}} \norm{\nabla f_i (\vx_{i, k}) - \nabla f_i (\vx_{i, k - 1})}.
	\end{equation*}
	Using $\|\nabla f_i\| \leq \bar{L}_0$, $|y_{i, k}| \leq 1$, and the $L$-smoothness of $f_i$ proves \eqref{eq:G_variation_sync}. 
\end{proof}

Since $\vy_k$ is computed using the previous primal iterate and the previous regularization parameter, we further need the following lagged dual mismatch bound.

\begin{lemma}
	\label{lem:dual_mismatch_stale}
	Suppose that Assumptions~\ref{asmp:smooth} and~\ref{asmp:graph} hold. 
	Then, for every $k \geq 0$ such that $\mathscr{H}_{\cK} (k)$ holds, we have
	\begin{equation}
		\label{eq:dual_mismatch_sq}
		\begin{aligned}
			\abs{\nabla_\nu h_{\rho_k}(\bar\nu_k,\bar\vx_k) - n \bar{s}_k}^2
			\leq {} &
			\frac{4n}{\rho_k^2}
			\dkh{
				\norm{\vE_{\nu,k}}^2
				+\bar L_0^2\norm{\vE_{x,k-1}}^2
			}\\
			&
			+4n^2
			\dkh{
				\frac{\bar L_0^2\eta_{x,k-1}^2}{\rho_k^2}\norm{\bar\vd_k}^2
				+\delta_k^2
			},
		\end{aligned}
	\end{equation}
	where $\delta_k$ is defined in Lemma~\ref{lem:elementary_variations}. 
\end{lemma}

\begin{proof}
	At $k=0$, the initialization of Algorithm~\ref{alg:defair_pdc} gives $\vE_{\nu, 0} = \vE_{x, -1} = 0$, $\vy_0 = \vy_{\rho_0} (\bar{\nu}_0, \bar{\vx}_0)$, and $n \bar{s}_0 = \nabla_\nu h_{\rho_0} (\bar{\nu}_0, \bar{\vx}_0)$. 
	Hence, the left-hand side of \eqref{eq:dual_mismatch_sq} is zero, and the asserted estimate follows immediately.

	Now we fix $k \geq 1$. 
	Using \eqref{eq:avg_s_identity} and the formula for $\nabla_\nu h_{\rho_k}$, we have
	\begin{equation*}
		n \bar{s}_k = 1 - \mathbf{1}_n^\top \vy_k,
		\quad\mbox{and}\quad
		\nabla_\nu h_{\rho_k} (\bar{\nu}_k, \bar{\vx}_k)
		= 1 - \mathbf{1}_n^\top \vy_{\rho_k} (\bar{\nu}_k, \bar{\vx}_k).
	\end{equation*}
	Then it can be readily verified that
	\begin{equation*}
		\begin{aligned}
			\abs{\nabla_\nu h_{\rho_k} (\bar\nu_k, \bar\vx_k) - n \bar{s}_k}
			\leq {} &
			\sqrt{n} \norm{\vy_k - \tilde{\vy}_k}
			+\sqrt{n} \norm{ \tilde{\vy}_k - \vy_{\rho_k} (\bar{\nu}_k, \bar{\vx}_{k - 1})} \\
			& + \sqrt{n} \norm{\vy_{\rho_k} (\bar{\nu}_k, \bar{\vx}_{k - 1}) - \vy_{\rho_k} (\bar{\nu}_k, \bar{\vx}_k)},
		\end{aligned}
	\end{equation*}
	where $\tilde{\vy}_k$ is defined in \eqref{eq:tilde-y}. 
	The first term is bounded by \eqref{eq:dist-tilde-y}. 
	The nonexpansiveness of $\Pi_{[0,1]^n}$ and the $\bar{L}_0$-Lipschitz continuity of $f_i$ give
	\begin{equation*}
		\norm{\tilde{\vy}_k - \vy_{\rho_k} (\bar{\nu}_k, \bar{\vx}_{k - 1})}
		\leq \frac{1}{\rho_k} \norm{\vE_{\nu, k}} 
		+ \frac{\bar{L}_0}{\rho_k} \norm{\vE_{x, k - 1}}.
	\end{equation*}
	Finally, it follows from $\bar{\vx}_k - \bar{\vx}_{k - 1} = - \eta_{x, k - 1} \bar{\vd}_k$ that
	\begin{equation*}
		\norm{ \vy_{\rho_k} (\bar{\nu}_k, \bar{\vx}_{k - 1}) - \vy_{\rho_k} (\bar{\nu}_k, \bar{\vx}_k) }
		\leq
		\frac{\bar{L}_0 \sqrt{n} \eta_{x, k - 1}}{\rho_k} \norm{\bar{\vd}_k}.
	\end{equation*}
	Combining these estimates and applying Young's inequality proves \eqref{eq:dual_mismatch_sq}.
\end{proof}

\subsection{Averaged descent}

\label{subsec:one_step_descent}

The purpose of this subsection is to derive a descent property for the regularized dual function along the averaged iterates. 
Our analysis balances negative gradient terms against positive network errors and regularization drifts.

We begin by examining the descent generated by the averaged dual update. 
This estimate forms the dual component of the overall descent argument.

\begin{lemma}
	\label{lem:dual_half_step_sync}
	Suppose that Assumptions~\ref{asmp:smooth} and~\ref{asmp:graph} hold. 
	Let the algorithmic parameters satisfy \eqref{eq:gc_rho_schedule} with $0 < \alpha \leq 1 / 4$ and $\tau \geq 1$. 
	Then it holds that
	\begin{equation}
		\label{eq:dual_half_step_sync}
		\begin{aligned}
			h_{\rho_k}(\bar{\nu}_{k+1},\bar{\vx}_k)
			\leq {} &
			h_{\rho_k}(\bar{\nu}_k,\bar{\vx}_k)
			-
			\frac{\eta_{\nu,k}}{2n}
			\bigl|\nabla_\nu h_{\rho_k}(\bar{\nu}_k,\bar{\vx}_k)\bigr|^2\\
			&
			+
			\frac{5\eta_{\nu,k}}{4n}
			\bigl|\nabla_\nu h_{\rho_k}(\bar{\nu}_k,\bar{\vx}_k)-n\bar{s}_k\bigr|^2,
		\end{aligned}
	\end{equation}
	for any $k \geq 0$. 
\end{lemma}

\begin{proof}
	Using \eqref{eq:average_updates} and the $n/\rho_k$-smoothness of $h_{\rho_k}$
	with respect to $\nu$ gives
	\begin{equation*}
		h_{\rho_k}(\bar{\nu}_{k+1},\bar{\vx}_k)
		\le
		h_{\rho_k}(\bar{\nu}_k,\bar{\vx}_k)
		-
		\eta_{\nu,k}\langle \nabla_\nu h_{\rho_k}(\bar{\nu}_k,\bar{\vx}_k),\bar{s}_k\rangle
		+
		\frac{n\eta_{\nu,k}^2}{2\rho_k}|\bar{s}_k|^2.
	\end{equation*}
	Write
	\begin{equation*}
		n\bar{s}_k
		=
		\nabla_\nu h_{\rho_k}(\bar{\nu}_k,\bar{\vx}_k)
		-
		\bigl(
		\nabla_\nu h_{\rho_k}(\bar{\nu}_k,\bar{\vx}_k)-n\bar{s}_k
		\bigr).
	\end{equation*}
	Expanding the inner product and applying Young's inequality yield
	\begin{equation*}
		\begin{aligned}
			h_{\rho_k}(\bar{\nu}_{k+1},\bar{\vx}_k)-h_{\rho_k}(\bar{\nu}_k,\bar{\vx}_k)
			\leq {} &
			-\frac{\eta_{\nu,k}}{n}
			\left(\frac34-\frac{\eta_{\nu,k}}{\rho_k}\right)
			\bigl|\nabla_\nu h_{\rho_k}(\bar{\nu}_k,\bar{\vx}_k)\bigr|^2\\
			&
			+
			\frac{\eta_{\nu,k}}{n}
			\left(1+\frac{\eta_{\nu,k}}{\rho_k}\right)
			\bigl|\nabla_\nu h_{\rho_k}(\bar{\nu}_k,\bar{\vx}_k)-n\bar{s}_k\bigr|^2.
		\end{aligned}
	\end{equation*}
	The condition $\alpha \leq 1 / 4$ gives $\eta_{\nu,k}\le\rho_k/4$, which proves \eqref{eq:dual_half_step_sync}. 
\end{proof}

We next turn to the averaged primal update. 
The influence of tracking errors and the interaction with dual updates are explicitly quantified.

\begin{lemma}
	\label{lem:primal_half_step_sync}
	Suppose that Assumptions~\ref{asmp:smooth} and~\ref{asmp:graph} hold. 
	Let the algorithmic parameters satisfy \eqref{eq:gc_rho_schedule} with $0 < \alpha \leq 1 / (2 (\rho_0 L + \bar{L}_0^2))$ and $\tau \geq 1$. 
	For every $k \geq 0$, if $\mathscr{H}_{\cK} (k + 1)$ holds, we have
	\begin{equation}
		\label{eq:primal_half_step_sync}
		\begin{aligned}
			h_{\rho_k}(\bar{\nu}_{k+1},\bar{\vx}_{k+1})
			\leq {} &
			h_{\rho_k}(\bar{\nu}_{k+1},\bar{\vx}_k)
			-
			\frac{\eta_{x,k}}{4n}
			\bigl\|\nabla_{\vx} h_{\rho_k}(\bar{\nu}_k,\bar{\vx}_k)\bigr\|^2 \\
			&
			-
			\frac{n\eta_{x,k}}{4}\|\bar{\vd}_{k+1}\|^2
			+
			\frac{\eta_{x,k}}{2n}
			\bigl\|\nabla_{\vx} h_{\rho_k}(\bar{\nu}_{k+1},\bar{\vx}_k)-n\bar{\vd}_{k+1}\bigr\|^2 \\
			&
			+
			\frac{\bar L_0^2\eta_{x,k}\eta_{\nu,k}^2}{n\rho_k^2}
			\bigl|\nabla_\nu h_{\rho_k}(\bar{\nu}_k,\bar{\vx}_k)\bigr|^2 \\
			&
			+
			\frac{\bar L_0^2\eta_{x,k}\eta_{\nu,k}^2}{n\rho_k^2}
			\bigl|\nabla_\nu h_{\rho_k}(\bar{\nu}_k,\bar{\vx}_k)-n\bar{s}_k\bigr|^2.
		\end{aligned}
	\end{equation}
\end{lemma}

\begin{proof}
	Using \eqref{eq:average_updates} and the
	$n(L+\bar L_0^2/\rho_k)$-smoothness of $h_{\rho_k}$ with respect to $\vx$ gives
	\begin{equation*}
		\begin{aligned}
			h_{\rho_k}(\bar{\nu}_{k+1},\bar{\vx}_{k+1})
			\leq {} &
			h_{\rho_k}(\bar{\nu}_{k+1},\bar{\vx}_k)
			-
			\eta_{x,k}
			\langle \nabla_{\vx} h_{\rho_k}(\bar{\nu}_{k+1},\bar{\vx}_k),\bar{\vd}_{k+1}\rangle\\
			&
			+
			\frac{n(L+\bar L_0^2/\rho_k)\eta_{x,k}^2}{2}\|\bar{\vd}_{k+1}\|^2.
		\end{aligned}
	\end{equation*}
	Using the identity
	\begin{equation*}
		\begin{aligned}
			-\eta_{x,k}
			\langle \nabla_{\vx} h_{\rho_k}(\bar{\nu}_{k+1},\bar{\vx}_k),\bar{\vd}_{k+1}\rangle
			= {} &
			-
			\frac{\eta_{x,k}}{2n}
			\|\nabla_{\vx} h_{\rho_k}(\bar{\nu}_{k+1},\bar{\vx}_k)\|^2
			-
			\frac{n\eta_{x,k}}{2}
			\|\bar{\vd}_{k+1}\|^2\\
			&
			+
			\frac{\eta_{x,k}}{2n}
			\|\nabla_{\vx} h_{\rho_k}(\bar{\nu}_{k+1},\bar{\vx}_k)-n\bar{\vd}_{k+1}\|^2,
		\end{aligned}
	\end{equation*}
	we obtain
	\begin{equation}\label{eq:primal_use_main}
		\begin{aligned}
			h_{\rho_k}(\bar{\nu}_{k+1},\bar{\vx}_{k+1})
			\leq {} &
			h_{\rho_k}(\bar{\nu}_{k+1},\bar{\vx}_k)
			-
			\frac{\eta_{x,k}}{2n}
			\|\nabla_{\vx} h_{\rho_k}(\bar{\nu}_{k+1},\bar{\vx}_k)\|^2\\
			&
			+
			\frac{\eta_{x,k}}{2n}
			\|\nabla_{\vx} h_{\rho_k}(\bar{\nu}_{k+1},\bar{\vx}_k)-n\bar{\vd}_{k+1}\|^2\\
			&
			-
			\frac{n\eta_{x,k}}{2}
			\left(
			1-\eta_{x,k}\Bigl(L+\frac{\bar L_0^2}{\rho_k}\Bigr)
			\right)\|\bar{\vd}_{k+1}\|^2.
		\end{aligned}
	\end{equation}
	By the stepsize condition, the last term is bounded above by $- n \eta_{x,k} \|\bar{\vd}_{k+1}\|^2 / 4$.

	The update rule $\bar{\nu}_{k + 1} - \bar{\nu}_k = - \eta_{\nu, k} \bar{s}_k$ together with the $n\bar L_0/\rho_k$-Lipschitz continuity of $\nabla_{\vx} h_{\rho_k}$ in $\nu$ gives
	\begin{equation}
		\label{eq:primal_use_2}
		\begin{aligned}
			-
			\frac{\eta_{x,k}}{2n}
			\|\nabla_{\vx} h_{\rho_k}(\bar{\nu}_{k+1},\bar{\vx}_k)\|^2
			\leq {} &
			-
			\frac{\eta_{x,k}}{4n}
			\|\nabla_{\vx} h_{\rho_k}(\bar{\nu}_k,\bar{\vx}_k)\|^2\\
			&
			+
			\frac{n\bar L_0^2\eta_{x,k}\eta_{\nu,k}^2}{2\rho_k^2}|\bar{s}_k|^2.
		\end{aligned}
	\end{equation}
	Since $n\bar{s}_k = \nabla_\nu h_{\rho_k} (\bar{\nu}_k, \bar{\vx}_k) - (\nabla_\nu h_{\rho_k} (\bar{\nu}_k, \bar{\vx}_k) - n \bar{s}_k)$, Young's inequality yields
	\begin{equation}
		\label{eq:primal_use_3}
		|\bar{s}_k|^2
		\le
		\frac{2}{n^2}
		\bigl|\nabla_\nu h_{\rho_k}(\bar{\nu}_k,\bar{\vx}_k)\bigr|^2
		+
		\frac{2}{n^2}
		\bigl|\nabla_\nu h_{\rho_k}(\bar{\nu}_k,\bar{\vx}_k)-n\bar{s}_k\bigr|^2.
	\end{equation}
	Substituting \eqref{eq:primal_use_2} and \eqref{eq:primal_use_3} into \eqref{eq:primal_use_main} gives \eqref{eq:primal_half_step_sync}.
\end{proof}

The previous descent estimates are derived under a fixed regularization parameter. 
Since the regularization parameter decreases over the iterations, it remains to control the drift induced by this change.

\begin{lemma}
	\label{lem:rho_drift_sync}
	For any $\nu \in \R$, $\vx \in \R^m$, and every $k \geq 0$, it holds that
	\begin{equation*}
		\abs{ h_{\rho_{k + 1}} (\nu, \vx) - h_{\rho_k} (\nu, \vx) }
		\leq
		\frac{n}{2} \abs{ \rho_{k + 1} - \rho_k }.
	\end{equation*}
\end{lemma}

\begin{proof}
	For fixed $t$, we define $\phi_t (\rho) := \rho \psi (t / \rho)$. 
	Then it holds that $\nabla \phi_t (\rho) = \psi (t / \rho) - (t / \rho) \nabla \psi (t / \rho)$. 
	By the definition of $\psi$, we have $\sup_{t \in \R} \abs{ \psi(t) - t \nabla \psi (t) } \leq 1 / 2$, and hence
	$|\nabla \phi_t (\rho)| \leq 1 / 2$. 
	Applying the mean value theorem to each term with $t = f_i (\vx) - \nu$ gives
	\begin{equation*}
		\abs{ h_{\rho_{k + 1}} (\nu, \vx) - h_{\rho_k} (\nu, \vx) }
		\leq
		\sum_{i = 1}^n
		\frac{1}{2} \abs{ \rho_{k + 1} - \rho_k }
		=
		\frac{n}{2} \abs{ \rho_{k + 1} - \rho_k }.
	\end{equation*}
	The proof is completed. 
\end{proof}

We are now in a position to combine the descent estimates for the primal and dual updates with the regularization drift. 
The resulting inequality serves as the cornerstone of the Lyapunov argument developed later. 

\begin{lemma}
	\label{lem:total_descent}
	Let Assumptions~\ref{asmp:smooth} and~\ref{asmp:graph} hold. 
	Let the algorithmic parameters satisfy \eqref{eq:gc_rho_schedule} with $0 < \alpha \leq \min \{1 / 4, 1 / (2 \bar{L}_0), 1 / (2 (\rho_0 L + \bar{L}_0^2))\}$ and $\tau \geq 1$. 
	Then, for every $k \geq 0$ such that $\mathscr{H}_{\cK} (k + 1)$ holds, we have
	\begin{equation}
		\label{eq:total_one_step_descent}
		\begin{aligned}
			h_{\rho_{k + 1}} (\bar{\nu}_{k + 1}, \bar{\vx}_{k + 1})
			- h_{\rho_k} (\bar{\nu}_k, \bar{\vx}_k)
			\leq {} &
			- \frac{\eta_{x, k}}{4 n}
			\norm{\nabla_{\vx} h_{\rho_k} (\bar{\nu}_k, \bar{\vx}_k)}^2
			- \frac{\eta_{\nu, k}}{4 n}
			\abs{\nabla_\nu h_{\rho_k} (\bar{\nu}_k, \bar{\vx}_k)}^2 \\
			&
			- \frac{n \eta_{x, k}}{4} \norm{\bar{\vd}_{k + 1}}^2
			+ \frac{6 n \bar{L}_0^2 \eta_{\nu, k} \eta_{x, k - 1}^2}{\rho_k^2} \norm{\bar{\vd}_k}^2\\
			&
			+ \frac{\bar{L}_0^2 \eta_{x, k}}{\rho_k^2} \norm{\vE_{\nu, k + 1}}^2
			+ \eta_{x, k} \dkh{ L + \frac{\bar{L}_0^2}{\rho_k} }^2 \norm{\vE_{x, k}}^2 \\
			&
			+ \frac{6 \eta_{\nu, k}}{\rho_k^2} \norm{\vE_{\nu, k}}^2
			+ \frac{6 \bar{L}_0^2 \eta_{\nu, k}}{\rho_k^2} \norm{\vE_{x, k - 1}}^2 \\
			&
			+ 6 n \eta_{\nu, k} \delta_k^2
			+ \frac{n}{2} \abs{\rho_{k + 1} - \rho_k}.
		\end{aligned}
	\end{equation}
\end{lemma}

\begin{proof}
	Adding the estimates in Lemmas~\ref{lem:dual_half_step_sync}, \ref{lem:primal_half_step_sync}, and~\ref{lem:rho_drift_sync} yields
	\begin{equation}
		\label{eq:total_descent_intermediate}
		\begin{aligned}
			& h_{\rho_{k+1}}(\bar{\nu}_{k+1},\bar{\vx}_{k+1})
			-
			h_{\rho_k}(\bar{\nu}_k,\bar{\vx}_k)\\
			\leq {} &
			-
			\frac{\eta_{x,k}}{4n}
			\|\nabla_{\vx} h_{\rho_k}(\bar{\nu}_k,\bar{\vx}_k)\|^2
			-
			\frac{\eta_{\nu,k}}{2n}
			|\nabla_\nu h_{\rho_k}(\bar{\nu}_k,\bar{\vx}_k)|^2
			-
			\frac{n\eta_{x,k}}{4}\|\bar{\vd}_{k+1}\|^2\\
			&
			+
			\frac{\eta_{x,k}}{2n}
			\|\nabla_{\vx} h_{\rho_k}(\bar{\nu}_{k+1},\bar{\vx}_k)-n\bar{\vd}_{k+1}\|^2\\
			&
			+
			\frac{5\eta_{\nu,k}}{4n}
			|\nabla_\nu h_{\rho_k}(\bar{\nu}_k,\bar{\vx}_k)-n\bar{s}_k|^2\\
			&
			+
			\frac{\bar L_0^2\eta_{x,k}\eta_{\nu,k}^2}{n\rho_k^2}
			|\nabla_\nu h_{\rho_k}(\bar{\nu}_k,\bar{\vx}_k)|^2
			+
			\frac{\bar L_0^2\eta_{x,k}\eta_{\nu,k}^2}{n\rho_k^2}
			|\nabla_\nu h_{\rho_k}(\bar{\nu}_k,\bar{\vx}_k)-n\bar{s}_k|^2\\
			&
			+
			\frac{n}{2}|\rho_{k+1}-\rho_k|.
		\end{aligned}
	\end{equation}
	The condition $\alpha \leq 1 / (2 \bar{L}_0)$ gives $\eta_{x,k}\eta_{\nu,k}\le\rho_k^2/(4\bar L_0^2)$, which implies that
	\begin{equation*}
		-\frac{\eta_{\nu,k}}{2n}
		+ \frac{\bar L_0^2\eta_{x,k}\eta_{\nu,k}^2}{n\rho_k^2}
		\le -\frac{\eta_{\nu,k}}{4n},
		\quad\mbox{and}\quad
		\frac{5\eta_{\nu,k}}{4n}
		+ \frac{\bar L_0^2\eta_{x,k}\eta_{\nu,k}^2}{n\rho_k^2}
		\le \frac{3\eta_{\nu,k}}{2n}.
	\end{equation*}
	Substituting these bounds into \eqref{eq:total_descent_intermediate} gives
	\begin{equation}\label{eq:total_descent_simplified}
		\begin{aligned}
			& h_{\rho_{k+1}}(\bar{\nu}_{k+1},\bar{\vx}_{k+1})
			- h_{\rho_k}(\bar{\nu}_k,\bar{\vx}_k)\\
			\leq {} &
			- \frac{\eta_{x,k}}{4n}
			\|\nabla_{\vx} h_{\rho_k}(\bar{\nu}_k,\bar{\vx}_k)\|^2
			- \frac{\eta_{\nu,k}}{4n}
			|\nabla_\nu h_{\rho_k}(\bar{\nu}_k,\bar{\vx}_k)|^2
			- \frac{n\eta_{x,k}}{4}\|\bar{\vd}_{k+1}\|^2\\
			&
			+ \frac{\eta_{x,k}}{2n}
			\|\nabla_{\vx} h_{\rho_k}(\bar{\nu}_{k+1},\bar{\vx}_k)-n\bar{\vd}_{k+1}\|^2\\
			&
			+ \frac{3\eta_{\nu,k}}{2n}
			|\nabla_\nu h_{\rho_k}(\bar{\nu}_k,\bar{\vx}_k)-n\bar{s}_k|^2
			+ \frac{n}{2}|\rho_{k+1}-\rho_k|.
		\end{aligned}
	\end{equation}
	Lemma~\ref{lem:mismatch} gives
	\begin{equation}
		\label{eq:primal_mismatch_sq}
		\norm{\nabla_{\vx} h_{\rho_k}(\bar{\nu}_{k+1},\bar{\vx}_k)-n\bar{\vd}_{k+1}}^2
		\leq
		\frac{2n\bar L_0^2}{\rho_k^2}\norm{\vE_{\nu,k+1}}^2
		+ 2n\Bigl(L+\frac{\bar L_0^2}{\rho_k}\Bigr)^2\norm{\vE_{x,k}}^2.
	\end{equation}
	Combining \eqref{eq:primal_mismatch_sq} with \eqref{eq:dual_mismatch_sq} in \eqref{eq:total_descent_simplified} gives \eqref{eq:total_one_step_descent}.
\end{proof}

\subsection{Tracking errors}

\label{subsec:tracking}

We move on to derive upper bounds on the tracking errors arising from the descent estimate of the regularized dual function. 
The ensuing analysis involves several constants that are somewhat intricate in form. 
To facilitate the narrative, we present their definitions below. 
\begin{equation}
	\label{eq:gc_basic_consts}
	\left\{
	\begin{aligned}
		& \bar\rho:=\frac{1+\rho_W^2}{2}\in(0,1),
		\qquad
		\kappa_W:=\frac{1+\rho_W^2}{1-\rho_W^2},
		\qquad
		\sigma:=\frac{1-\bar\rho}{4}, \\
		& \tau_\star
		:=
		\max\left\{
		1,
		\frac{(\bar\rho+\sigma)^3}{1-(\bar\rho+\sigma)^3}
		\right\},
		\qquad
		\theta
		:=
		\sup_{k\geq0}\frac{\rho_{k-1}}{\rho_k}
		=
		\left(\frac{\tau+1}{\tau}\right)^{1/3},\\
		& B:=\kappa_W\max\Bigl\{
		9+18\bar L_0^2,
		\;72+144\bar L_0^2+(9+18\bar L_0^2)(1+\rho_W)^2,\\
		& \hspace{66pt}9\bar L_0^2\theta^2+6(L^2\rho_0^2+3\bar L_0^4\theta^2),
		\;(18+36\bar L_0^2)/n,\\
		& \hspace{66pt}72\bar L_0^2+144\bar L_0^4
		+ (9 \bar{L}_0^2 + 18 \bar{L}_0^4 + 6 L^2 \rho_0^2) (1+\rho_W)^2,\\
		& \hspace{66pt}n(81\bar L_0^2\theta^2+6L^2\rho_0^2+162\bar L_0^4\theta^2),
		\;75n+150n\bar L_0^2
		\Bigr\}.
	\end{aligned}
	\right.
\end{equation}
These constants are independent of the iteration counter and do not enter the implementation of Algorithm~\ref{alg:defair_pdc}.

\begin{lemma}
	\label{lem:gc_error_closure}
	Suppose that Assumptions~\ref{asmp:smooth} and~\ref{asmp:graph} hold. 
	Let the algorithmic parameters satisfy \eqref{eq:gc_rho_schedule} with $0 < \alpha \leq 1$ and $\tau \geq \tau_\star$. 
	Then, for every $k \geq 0$ such that $\mathscr{H}_{\cK} (k)$ holds, we have
	\begin{equation}
		\label{eq:gc_Es_compact}
		\begin{aligned}
			\norm{\vE_{s, k + 1}}^2
			\leq {} & \dkh{\bar{\rho} + 9 \kappa_W \alpha^2} \norm{\vE_{s, k}}^2
			+ 9 \kappa_W \bar{L}_0^2 \theta^2 \alpha^2 \norm{\vE_{d, k}}^2  \\
			& 
			+ \frac{B}{\rho_k^2} \norm{\vE_{\nu, k}}^2
			+ \frac{B}{\rho_k^2} \norm{\vE_{x, k - 1}}^2
			+ B \alpha^2 \norm{\bar{\vd}_k}^2 \\
			& + B \alpha^2 \abs{\nabla_\nu h_{\rho_k} (\bar{\nu}_k, \bar{\vx}_k)}^2
			+ B \delta_k^2,
		\end{aligned}  
	\end{equation}
	and
	\begin{equation}
		\label{eq:gc_Ed_compact}
		\begin{aligned}
			\norm{\vE_{d, k + 1}}^2
			\leq {} & \dkh{ \bar{\rho} + 6 \kappa_W (L^2 \rho_0^2 + 3 \bar{L}_0^4 \theta^2) \alpha^2 } \norm{\vE_{d, k}}^2
			+ 18 \kappa_W \bar{L}_0^2 \alpha^2 \norm{\vE_{s, k}}^2 \\
			& + \frac{B}{\rho_k^2} \norm{\vE_{\nu, k}}^2
			+ \frac{B}{\rho_k^2} \norm{\vE_{x, k - 1}}^2
			+ B \alpha^2 \norm{\bar{\vd}_k}^2 \\
			& + B \alpha^2 \abs{\nabla_\nu h_{\rho_k} (\bar{\nu}_k, \bar{\vx}_k)}^2
			+ B \delta_k^2. 
		\end{aligned}
	\end{equation}  
\end{lemma}

\begin{proof}
	Applying Young's inequality to \eqref{eq:err_s_bound} and \eqref{eq:err_d_bound}, with $\bar\rho$ and $\kappa_W$ defined in \eqref{eq:gc_basic_consts}, gives
	\begin{equation}
		\label{eq:gc_start_EsEd}
		\|\vE_{s,k+1}\|^2
		\le \bar\rho\|\vE_{s,k}\|^2 + \kappa_W\|\vy_{k+1}-\vy_k\|^2,
	\end{equation}
	and
	\begin{equation}
		\label{eq:gc_start_Ed}
		\|\vE_{d,k+1}\|^2
		\le \bar\rho\|\vE_{d,k}\|^2 + \kappa_W\|\vG_{k+1}-\vG_k\|^2.
	\end{equation}
	Two relationships \eqref{eq:y_variation_sync} and \eqref{eq:G_variation_sync} give
	\begin{equation}
		\label{eq:gc_y_variation_sq}
		\|\vy_{k+1}-\vy_k\|^2
		\le
		\frac{3\bar L_0^2}{\rho_k^2}\|\vX_k-\vX_{k-1}\|^2
		+
		\frac{3}{\rho_k^2}\|\boldsymbol{\nu}_{k+1}-\boldsymbol{\nu}_k\|^2
		+
		3n\delta_k^2,
	\end{equation}
	and
	\begin{equation}
		\label{eq:gc_G_variation_sq}
		\|\vG_{k+1}-\vG_k\|^2
		\le
		2\bar L_0^2\|\vy_{k+1}-\vy_k\|^2
		+
		2L^2\|\vX_k-\vX_{k-1}\|^2.
	\end{equation}
	For $k \geq 1$, squaring \eqref{eq:X_increment_sync} at index $k - 1$ and \eqref{eq:nu_increment_sync} at index $k$, and using $\eta_{x, k - 1} = \alpha \rho_{k - 1}$ and $\eta_{\nu, k} = \alpha \rho_k$, yields
	\begin{equation}
		\label{eq:gc_X_increment_sq}
		\|\vX_k-\vX_{k-1}\|^2
		\le
		3(1+\rho_W)^2\|\vE_{x,k-1}\|^2
		+
		3\alpha^2\rho_{k-1}^2\|\vE_{d,k}\|^2
		+
		3n\alpha^2\rho_{k-1}^2\|\bar{\vd}_k\|^2,
	\end{equation}
	and
	\begin{equation}
		\label{eq:gc_nu_increment_sq}   
		\|\boldsymbol{\nu}_{k+1}-\boldsymbol{\nu}_k\|^2
		\le
		3(1+\rho_W)^2\|\vE_{\nu,k}\|^2
		+
		3\alpha^2\rho_k^2\|\vE_{s,k}\|^2
		+
		3n\alpha^2\rho_k^2|\bar{s}_k|^2.
	\end{equation}
	At $k = 0$, the estimates \eqref{eq:gc_X_increment_sq} and \eqref{eq:gc_nu_increment_sq} hold trivially. 
	From \eqref{eq:dual_mismatch_sq}, $\eta_{x, k - 1} = \alpha \rho_{k - 1}$, $0 < \alpha \leq 1$, and $\rho_{k - 1} / \rho_k \leq \theta$, we obtain
	\begin{equation}
		\label{eq:gc_sbar_bound}
		\begin{aligned}
			|\bar s_k|^2
			\leq {} &
			\frac{2}{n^2}\bigl|\nabla_\nu h_{\rho_k}(\bar{\nu}_k,\bar{\vx}_k)\bigr|^2
			+
			\frac{8}{n\rho_k^2}\|\vE_{\nu,k}\|^2
			+
			\frac{8\bar L_0^2}{n\rho_k^2}\|\vE_{x,k-1}\|^2\\
			&
			+
			8\bar L_0^2\theta^2 \alpha^2\|\bar{\vd}_k\|^2
			+
			8\delta_k^2.
		\end{aligned}
	\end{equation}
	Substituting \eqref{eq:gc_sbar_bound} into \eqref{eq:gc_nu_increment_sq} gives
	\begin{equation}
		\label{eq:gc_nu_increment_sq_2}
		\begin{aligned}
			\|\boldsymbol{\nu}_{k+1}-\boldsymbol{\nu}_k\|^2
			\leq {} &
			\bigl(3(1+\rho_W)^2+24\bigr)\|\vE_{\nu,k}\|^2
			+
			3\alpha^2\rho_k^2\|\vE_{s,k}\|^2\\
			&
			+
			24\bar L_0^2\alpha^2\|\vE_{x,k-1}\|^2
			+
			24n\bar L_0^2\theta^2 \alpha^4\rho_k^2\|\bar{\vd}_k\|^2\\
			&
			+
			\frac{6\alpha^2\rho_k^2}{n}\bigl|\nabla_\nu h_{\rho_k}(\bar{\nu}_k,\bar{\vx}_k)\bigr|^2
			+
			24n\alpha^2\rho_k^2\delta_k^2.
		\end{aligned}
	\end{equation}
	Substituting \eqref{eq:gc_X_increment_sq} and \eqref{eq:gc_nu_increment_sq_2} into \eqref{eq:gc_y_variation_sq} gives
	\begin{equation}
		\label{eq:gc_y_variation_sq_2}
		\begin{aligned}
			& \|\vy_{k+1}-\vy_k\|^2 \\
			\leq {} &
			9\alpha^2\|\vE_{s,k}\|^2
			+
			9\bar L_0^2\theta^2 \alpha^2\|\vE_{d,k}\|^2\\
			&
			+
			\frac{1}{\rho_k^2}
			\dkh{
				\dkh{ 9(1+\rho_W)^2+72 } \|\vE_{\nu,k}\|^2
				+
				\dkh{ 9\bar L_0^2(1+\rho_W)^2+72\bar L_0^2 } \|\vE_{x,k-1}\|^2
			} \\
			&
			+
			81n\bar L_0^2\theta^2 \alpha^2\|\bar{\vd}_k\|^2
			+
			\frac{18\alpha^2}{n} \abs{ \nabla_\nu h_{\rho_k}(\bar{\nu}_k,\bar{\vx}_k) }^2
			+
			75n\delta_k^2.
		\end{aligned}
	\end{equation}
	Similarly, substituting \eqref{eq:gc_y_variation_sq_2} and \eqref{eq:gc_X_increment_sq} into \eqref{eq:gc_G_variation_sq}, and using $\rho_{k - 1} / \rho_k \leq \theta$ and $\alpha \leq 1$, yields
	\begin{equation}
		\label{eq:gc_G_variation_sq_2}
		\begin{aligned}
			\|\vG_{k+1}-\vG_k\|^2
			\leq {} &
			18\bar L_0^2\alpha^2\|\vE_{s,k}\|^2
			+
			\dkh{ 18\bar L_0^4\theta^2+6L^2\rho_0^2 } \alpha^2\|\vE_{d,k}\|^2\\
			&
			+
			\frac{1}{\rho_k^2}
			\dkh{ 
				\frac{B}{\kappa_W}\|\vE_{\nu,k}\|^2
				+
				\frac{B}{\kappa_W}\|\vE_{x,k-1}\|^2
			} \\
			&
			+
			\frac{B}{\kappa_W}\alpha^2\|\bar{\vd}_k\|^2
			+
			\frac{B}{\kappa_W}\alpha^2 \abs{ \nabla_\nu h_{\rho_k}(\bar{\nu}_k,\bar{\vx}_k) }^2
			+
			\frac{B}{\kappa_W}\delta_k^2.
		\end{aligned}
	\end{equation}
	The constant $B$ defined in \eqref{eq:gc_basic_consts} dominates all the coefficients in \eqref{eq:gc_y_variation_sq_2} and \eqref{eq:gc_G_variation_sq_2} that are collected in the final five terms of \eqref{eq:gc_Es_compact} and \eqref{eq:gc_Ed_compact}. 
	Substituting these two bounds into \eqref{eq:gc_start_EsEd} and \eqref{eq:gc_start_Ed} yields \eqref{eq:gc_Es_compact} and \eqref{eq:gc_Ed_compact}. 
\end{proof}

\subsection{Lyapunov argument}

\label{subsec:lyapunov}

This subsection closes the preceding estimates with the weighted Lyapunov function
\begin{equation*}
	\mathfrak{L}_k
	:=
	h_{\rho_k}(\bar\nu_k,\bar\vx_k)
	+
	\frac{M}{\rho_{k-1}}
	\dkh{\norm{\vE_{\nu,k}}^2+\norm{\vE_{x,k}}^2}
	+
	\beta M\alpha^2\rho_{k-1}
	\dkh{\norm{\vE_{s,k}}^2+\norm{\vE_{d,k}}^2},
\end{equation*}
where $\beta$ and $M$ are two positive constants defined by
\begin{equation}
	\label{eq:gc_M_choice}
	\beta:=\frac{4\kappa_W}{1-\bar\rho},
	\quad\mbox{and}\quad
	M:=\frac{16}{\sigma}
	\max\hkh{1,\;6+\bar L_0^2\bar\rho,\;(\rho_0L+\bar L_0^2)^2},
\end{equation}
respectively. 
To establish the desired result, we need to impose an upper bound on the parameter $\alpha$ appearing in \eqref{eq:gc_rho_schedule}, which is given by
\begin{equation}
	\label{eq:gc_a_lyap}
	\begin{aligned}
		\overline{\alpha}
		:= \min \Biggl\{
		&\frac{1}{4},
		\frac{1}{2 \bar{L}_0},
		\frac{1}{2 (\rho_0 L + \bar{L}_0^2)},
		\sqrt{\frac{1 - \bar{\rho}}{8 B}},
		\frac{1}{6 \sqrt{\kappa_W} \bar{L}_0},
		\frac{1}{\sqrt{12 \kappa_W (L^2 \rho_0^2 + 3 \bar{L}_0^4 \theta^2)}}, \\
		& \sqrt{\frac{\sigma}{8 \Gamma B}},
		\frac{M \sigma}{48 \theta \bar{L}_0^2},
		\sqrt{\frac{\sigma}{8 \theta \Gamma B}},
		\dkh{\frac{1}{8 n M \Gamma B}}^{1/3},
		\sqrt{\frac{n}{8 (6 n \bar{L}_0^2 \theta + M \Gamma B)}}
		\Biggr\},
	\end{aligned}
\end{equation}
with $\Gamma := 2 \beta + \kappa_W$.

\begin{lemma}
	\label{lem:gc_finite_horizon}
	Suppose that Assumptions~\ref{asmp:smooth} and~\ref{asmp:graph} hold. 
	Let the algorithmic parameters satisfy \eqref{eq:gc_rho_schedule} with $0 < \alpha \leq \overline{\alpha}$ and $\tau \geq \tau_\star$. 
	Then, for every integer $K \geq 1$ such that $\mathscr{H}_{\cK} (K)$ holds, we have
	\begin{equation}
		\label{eq:gc_finite_horizon}
		\begin{aligned}
			\mathfrak{C}
			\geq {} &
			\mathfrak{L}_K
			+
			\sum_{k = 0}^{K - 1}
			\Biggl[
			\frac{\alpha \rho_k}{4 n} \norm{\nabla_{\vx} h_{\rho_k} (\bar\nu_k, \bar\vx_k)}^2
			+
			\frac{\alpha \rho_k}{8 n} \abs{\nabla_\nu h_{\rho_k} (\bar\nu_k, \bar\vx_k)}^2 \\
			&\hspace{2cm}
			+
			\frac{n \alpha \rho_k}{8} \norm{\bar{\vd}_{k + 1}}^2
			+
			\frac{M \sigma}{4 \rho_k} \dkh{\norm{\vE_{\nu, k}}^2 + \norm{\vE_{x, k}}^2}\\
			&\hspace{2cm}
			+
			\frac{\beta M \alpha^2 (1 - \bar{\rho})}{8} \rho_k \dkh{\norm{\vE_{s, k}}^2 + \norm{\vE_{d, k}}^2}
			\Biggr],
		\end{aligned}
	\end{equation}
	where $\mathfrak{C}$ is a constant defined by
	\begin{equation*}
		\mathfrak{C}
		:=
		\mathfrak{L}_0
		+
		\widetilde{C} \rho_0 \norm{\bar{\vd}_0}^2
		+
		2 \rho_0 (6 n \alpha + M \alpha^2 \Gamma B) / 9
		+
		n \rho_0 / 2,
	\end{equation*}
	with $\widetilde{C} := \alpha^3 (6 n \bar{L}_0^2 \theta + M \alpha \Gamma B)$.
\end{lemma}

\begin{proof}
	For $k = 0$, we have $\rho_{-1} / \rho_0 = 1$. 
	For $k \geq 1$, it holds that $\rho_{k-1} / \rho_k \leq \theta$. 
	The choice $\tau \geq \tau_\star$ therefore yields $(\bar{\rho} + \sigma) \rho_{k - 1} / \rho_k \leq 1$, which implies that
	\begin{equation}
		\label{eq:gc_weight_margin}
		\frac{\bar{\rho}}{\rho_k}
		\leq
		\frac{1}{\rho_{k - 1}} - \frac{\sigma}{\rho_k},
	\end{equation}
	for all $k \geq 0$. 
	Substituting \eqref{eq:gc_rho_schedule} into \eqref{eq:total_one_step_descent}, and using $\rho_{k - 1} / \rho_k \leq \theta$, gives
	\begin{equation}
		\label{eq:gc_descent_sched_compact}
		\begin{aligned}
			h_{\rho_{k+1}}(\bar\nu_{k+1},\bar\vx_{k+1})
			-h_{\rho_k}(\bar\nu_k,\bar\vx_k)
			\leq {} &
			-\frac{\alpha\rho_k}{4n}\norm{\nabla_{\vx}h_{\rho_k}(\bar\nu_k,\bar\vx_k)}^2
			-\frac{\alpha\rho_k}{4n}\abs{\nabla_\nu h_{\rho_k}(\bar\nu_k,\bar\vx_k)}^2\\
			&
			-\frac{n\alpha\rho_k}{4}\norm{\bar\vd_{k+1}}^2
			+6n\bar L_0^2\theta \alpha^3\rho_{k-1}\norm{\bar\vd_k}^2\\
			&
			+\frac{\alpha\bar L_0^2}{\rho_k}\norm{\vE_{\nu,k+1}}^2
			+\frac{6\alpha}{\rho_k}\norm{\vE_{\nu,k}}^2\\
			&
			+\frac{\alpha(\rho_0L+\bar L_0^2)^2}{\rho_k}\norm{\vE_{x,k}}^2
			+\frac{6\bar L_0^2\alpha}{\rho_k}\norm{\vE_{x,k-1}}^2\\
			&
			+6n\alpha\rho_k\delta_k^2
			+\frac n2\abs{\rho_{k+1}-\rho_k}.
		\end{aligned}
	\end{equation}
	Using \eqref{eq:err_nu_bound}, \eqref{eq:err_x_bound}, \eqref{eq:gc_M_choice}, and \eqref{eq:gc_weight_margin}, we obtain
	\begin{equation}
		\label{eq:gc_nu_increment_compact}
		\begin{aligned}
			\frac{M}{\rho_k} \norm{\vE_{\nu, k + 1}}^2
			- \frac{M}{\rho_{k - 1}} \norm{\vE_{\nu, k}}^2
			+ \frac{\alpha \bar{L}_0^2}{\rho_k} \norm{\vE_{\nu, k + 1}}^2
			\leq {} & 2 \kappa_W M \alpha^2 \rho_k \norm{\vE_{s, k}}^2 \\
			& - \frac{M \sigma}{2 \rho_k} \norm{\vE_{\nu, k}}^2,
		\end{aligned}
	\end{equation}
	and
	\begin{equation}
		\label{eq:gc_x_increment_raw}
		\frac{M}{\rho_k} \norm{\vE_{x, k + 1}}^2
		- \frac{M}{\rho_{k - 1}} \norm{\vE_{x, k}}^2
		\leq
		- \frac{M \sigma}{\rho_k} \norm{\vE_{x, k}}^2
		+\kappa_W M \alpha^2 \rho_k \norm{\vE_{d, k + 1}}^2.
	\end{equation}
	Substituting \eqref{eq:gc_Ed_compact} into \eqref{eq:gc_x_increment_raw} gives
	\begin{equation}
		\label{eq:gc_x_increment_compact}
		\begin{aligned}
			\frac{M}{\rho_k}\norm{\vE_{x,k+1}}^2
			-\frac{M}{\rho_{k-1}}\norm{\vE_{x,k}}^2
			\leq {} &
			-\frac{M\sigma}{\rho_k}\norm{\vE_{x,k}}^2
			+\kappa_W\bar\rho M\alpha^2\rho_k\norm{\vE_{d,k}}^2\\
			&+18\kappa_W^2\bar L_0^2M\alpha^4\rho_k\norm{\vE_{s,k}}^2\\
			&+6\kappa_W^2(L^2\rho_0^2+3\bar L_0^4\theta^2)M\alpha^4\rho_k\norm{\vE_{d,k}}^2\\
			&+\frac{\kappa_WM\alpha^2B}{\rho_k}
			\dkh{\norm{\vE_{\nu,k}}^2+\norm{\vE_{x,k-1}}^2}\\
			&+\kappa_WM\alpha^4B\rho_k
			\dkh{\norm{\bar\vd_k}^2+\abs{\nabla_\nu h_{\rho_k}(\bar\nu_k,\bar\vx_k)}^2}\\
			&+\kappa_WM\alpha^2B\rho_k\delta_k^2.
		\end{aligned}
	\end{equation}
	For the tracking component, \eqref{eq:gc_Es_compact}, \eqref{eq:gc_Ed_compact}, and \eqref{eq:gc_basic_consts} imply
	\begin{equation*}
		\begin{aligned}
			\norm{\vE_{s,k+1}}^2+\norm{\vE_{d,k+1}}^2
			\leq {} &
			(\bar\rho+\alpha^2B)
			\dkh{\norm{\vE_{s,k}}^2+\norm{\vE_{d,k}}^2}
			+\frac{2B}{\rho_k^2}\norm{\vE_{\nu,k}}^2\\
			&+\frac{2B}{\rho_k^2}\norm{\vE_{x,k-1}}^2
			+2B\alpha^2\norm{\bar\vd_k}^2\\
			&+2B\alpha^2\abs{\nabla_\nu h_{\rho_k}(\bar\nu_k,\bar\vx_k)}^2
			+2B\delta_k^2.
		\end{aligned}
	\end{equation*}
	Multiplying by $\beta M \alpha^2 \rho_k$, subtracting $\beta M \alpha^2 \rho_{k - 1} (\|\vE_{s, k}\|^2 + \|\vE_{d, k}\|^2)$, and using $\rho_k \leq \rho_{k - 1}$ yield
	\begin{equation}
		\label{eq:gc_tracking_increment}
		\begin{aligned}
			&\beta M\alpha^2\rho_k
			\dkh{\norm{\vE_{s,k+1}}^2+\norm{\vE_{d,k+1}}^2}
			-\beta M\alpha^2\rho_{k-1}
			\dkh{\norm{\vE_{s,k}}^2+\norm{\vE_{d,k}}^2}\\
			\leq {} &
			-\beta M\alpha^2\rho_k(1-\bar\rho-\alpha^2B)
			\dkh{\norm{\vE_{s,k}}^2+\norm{\vE_{d,k}}^2}\\
			&+\frac{2\beta M\alpha^2B}{\rho_k}\norm{\vE_{\nu,k}}^2
			+\frac{2\beta M\alpha^2B}{\rho_k}\norm{\vE_{x,k-1}}^2\\
			&+2\beta M\alpha^4B\rho_k
			\dkh{\norm{\bar\vd_k}^2+\abs{\nabla_\nu h_{\rho_k}(\bar\nu_k,\bar\vx_k)}^2}
			+2\beta M\alpha^2B\rho_k\delta_k^2.
		\end{aligned}
	\end{equation}
	Adding \eqref{eq:gc_descent_sched_compact}, \eqref{eq:gc_nu_increment_compact}, \eqref{eq:gc_x_increment_compact}, and \eqref{eq:gc_tracking_increment}, and using the condition $\alpha \leq \overline{\alpha}$, gives
	\begin{equation}
		\label{eq:gc_one_step_energy_compact}
		\begin{aligned}
			\mathfrak{L}_{k + 1} - \mathfrak{L}_k
			\leq {} &
			-\frac{\alpha\rho_k}{4n}\norm{\nabla_{\vx}h_{\rho_k}(\bar\nu_k,\bar\vx_k)}^2
			-\frac{\alpha\rho_k}{8n}\abs{\nabla_\nu h_{\rho_k}(\bar\nu_k,\bar\vx_k)}^2
			- \frac{n\alpha\rho_k}{4} \norm{\bar\vd_{k+1}}^2\\
			&-\frac{M\sigma}{4\rho_k}\norm{\vE_{\nu,k}}^2
			-\frac{\beta M\alpha^2(1-\bar\rho)}{8}\rho_k
			\dkh{\norm{\vE_{s,k}}^2+\norm{\vE_{d,k}}^2}\\
			&+(6\bar L_0^2\alpha+M\alpha^2\Gamma B)
			\frac{\norm{\vE_{x,k-1}}^2}{\rho_k}\\
			&+(6n\bar L_0^2\theta \alpha^3+M\alpha^4\Gamma B)
			\rho_{k-1}\norm{\bar\vd_k}^2\\
			&-\frac{M\sigma}{2\rho_k}\norm{\vE_{x,k}}^2
			+(6n\alpha+M\alpha^2\Gamma B)\rho_k\delta_k^2
			+\frac n2\abs{\rho_{k+1}-\rho_k}.
		\end{aligned}
	\end{equation}
	We now sum \eqref{eq:gc_one_step_energy_compact} over $k$ from $0$ to $K - 1$. 
	Since $\vE_{x, -1} = \vE_{x, 0} = 0$ and $1 / \rho_k \leq \theta / \rho_{k - 1}$, it holds that
	\begin{equation}
		\label{eq:gc_shift_xlag}
		\begin{aligned}
			\sum_{k = 0}^{K - 1} (6\bar{L}_0^2 \alpha + M \alpha^2 \Gamma B) \frac{\norm{\vE_{x,k-1}}^2}{\rho_k}
			\leq {} &
			\theta (6\bar{L}_0^2 \alpha + M \alpha^2 \Gamma B) \sum_{k = 0}^{K - 2} \frac{\norm{\vE_{x, k}}^2}{\rho_k} \\
			\leq {} &
			\frac{M \sigma}{4} \sum_{k = 0}^{K - 2} \frac{\norm{\vE_{x, k}}^2}{\rho_k},
		\end{aligned}
	\end{equation}
	where the last inequality follows from the definition of $\overline{\alpha}$. 
	Hence, this term is absorbed by the negative consensus term in \eqref{eq:gc_one_step_energy_compact}. 
	Likewise, we have
	\begin{equation}
		\label{eq:gc_shift_dlag}
		\sum_{k = 0}^{K - 1}
		(6 n \bar{L}_0^2 \theta \alpha^3 + M \alpha^4 \Gamma B) \rho_{k - 1} \norm{\bar{\vd}_k}^2
		=
		\widetilde{C} \rho_0 \norm{\bar{\vd}_0}^2
		+
		\widetilde{C} \sum_{k = 0}^{K - 2} \rho_k \norm{\bar{\vd}_{k + 1}}^2.
	\end{equation}
	Since $\widetilde{C} \leq n \alpha / 8$, the interior sum is absorbed by the negative averaged-direction term in \eqref{eq:gc_one_step_energy_compact}. 
	For $k \geq 1$, the convexity inequality $(1 + t)^{- 1 / 3} \geq 1 - t / 3$ for $t \geq 0$ gives $\delta_k = 1 - (1 + 1 / (k - 1 + \tau))^{- 1 / 3} \leq 1 / (3 (k - 1 + \tau))$. 
	Since $\delta_0 = 0$ and $\rho_k \leq \rho_0$, we have
	\begin{equation}
		\label{eq:sum-delta}
		\sum_{k = 0}^{K - 1} \rho_k \delta_k^2 
		\leq
		\frac{\rho_0}{9} \sum_{j = 0}^{\infty} \frac{1}{(j + \tau)^2}
		\leq
		\frac{\rho_0}{9} \dkh{ \frac{1}{\tau^2} + \int_0^\infty \frac{dt}{(t + \tau)^2} }
		\leq
		\frac{2 \rho_0}{9}.
	\end{equation}
	Moreover, the monotonicity of $\{\rho_k\}$ gives
	\begin{equation}
		\label{eq:sum-rho}
		\sum_{k = 0}^{K - 1} \abs{\rho_{k + 1} - \rho_k}
		=\rho_0 - \rho_K
		\leq \rho_0.
	\end{equation}
	Summing \eqref{eq:gc_one_step_energy_compact} over $k$ from $0$ to $K - 1$, using \eqref{eq:gc_shift_xlag}, \eqref{eq:gc_shift_dlag}, \eqref{eq:sum-delta}, and \eqref{eq:sum-rho}, and retaining the boundary term at $k = 0$ yields \eqref{eq:gc_finite_horizon}.
\end{proof}

\subsection{Boundedness of iterates}

\label{subsec:localization}

The analysis so far has been carried out conditionally on the localization event $\mathscr{H}_{\cK} (q)$. 
We now remove this requirement by showing that the compact set $\cK$ contains all local iterates and their network average. 
The final admissible stepsize threshold is defined by
\begin{equation*}
	\alpha_\star := \min \hkh{\overline{\alpha}, \frac{1-\rho_W}{2 \rho_0 D}, \frac{1}{\rho_0} \dkh{ (\beta M + 6 n \bar{L}_0^2 \theta + M \Gamma B) P + 2 (6 n + M \Gamma B) / 9 }^{- 1}},
\end{equation*}
where $P := \max \{ \|\vE_{s, 0}\|^2 + \|\vE_{d, 0}\|^2, \|\bar{\vd}_0\|^2 \}$ and $D := \|\vE_{d, 0}\| + 2 \bar{L}_0 \sqrt{n} / (1 - \rho_W) + \bar{L}_0 + 1$ are two constants.

\begin{lemma}
	\label{lem:localization}
	Let Assumptions~\ref{asmp:smooth} and~\ref{asmp:graph} hold. 
	Let the algorithmic parameters satisfy \eqref{eq:gc_rho_schedule} with $0 < \alpha \leq \alpha_\star$ and $\tau \geq \tau_\star$. 
	Then $\mathscr{H}_{\cK} (q)$ holds for every $q \geq 0$. 
\end{lemma}

\begin{proof}
	We first record a uniform tracking bound. 
	Suppose that $\mathscr{H}_{\cK} (q - 1)$ holds. 
	According to the definition of $\vG_t$, we have $\|\vG_t\| \leq \sqrt{n} \bar{L}_0$ for $0 \leq t \leq q$, which implies that $\|\vG_{t + 1} - \vG_t\| \leq 2 \sqrt{n} \bar{L}_0$ for $0 \leq t \leq q - 1$.
	Then it follows from \eqref{eq:err_d_bound} that
	\begin{equation}
		\label{eq:localization_D_bound}
		\norm{\vE_{d, t}}
		\leq
		\norm{\vE_{d, 0}} + \frac{2 \sqrt{n} \bar{L}_0}{1 - \rho_W},
		\quad\mbox{and}\quad
		\norm{\vd_{i, t}} \leq D,
	\end{equation}
	where the second estimate uses $\|\bar{\vd}_t\| \leq \bar{L}_0$ and the definition of $D$.

	Note that $\vE_{x, 0} = 0$. 
	As long as the preceding iterates remain in $\cK$, \eqref{eq:err_x_bound} and \eqref{eq:localization_D_bound} further give
	\begin{equation}
		\label{eq:localization_Ex_rough}
		\norm{\vE_{x, t}}
		\leq
		\frac{\alpha \rho_0 D}{1 - \rho_W}
		\leq
		\frac{1}{2},
	\end{equation}
	where the last inequality follows from the condition $\alpha \leq (1 - \rho_W) / (2 \rho_0 D)$.

	Since $\Phi (\bar{\vx}_0) \leq h_{\rho_0} (\bar\nu_0, \bar\vx_0) + \rho_0 / 2 \leq \varsigma$, the level set $\cQ$ defined in \eqref{eq:localization_sets_fixed} contains $\bar{\vx}_0$. 
	The initialization of Algorithm~\ref{alg:defair_pdc} gives $\vx_{i, 0} = \bar{\vx}_0 \in \cQ \subseteq \cK$ for all $i \in \cN$. 
	Thus, $\mathscr{H}_{\cK} (0)$ holds.

	Now we prove the assertion of this lemma by contradiction. 
	Let $T \geq 1$ be the smallest integer for which $\mathscr{H}_{\cK} (T)$ fails. 
	Then $\mathscr{H}_{\cK} (T - 1)$ holds. 
	We claim that $\bar{\vx}_K \in \cQ$ and $\vx_{i, K} \in \cU$ for every $0 \leq K < T$. 
	This is immediate at $K = 0$. 
	For $1 \leq K < T$, Lemma~\ref{lem:gc_finite_horizon} is applicable, which gives
	\begin{equation*}
		\mathfrak{L}_K \leq \mathfrak{C},
		\quad\mbox{and}\quad
		h_{\rho_K} (\bar{\nu}_K, \bar{\vx}_K) \leq \mathfrak{C}.
	\end{equation*}
	Using \eqref{eq:Phi_rho_approx}, $\Phi_{\rho_K} (\bar{\vx}_K) = \min_\nu h_{\rho_K} (\nu, \bar{\vx}_K)$, and $\rho_K \leq \rho_0$, we obtain
	\begin{equation*}
		\Phi (\bar{\vx}_K)
		\leq
		h_{\rho_K} (\bar{\nu}_K, \bar{\vx}_K) + \frac{\rho_K}{2}
		\leq
		\mathfrak{C} + \frac{\rho_0}{2}.
	\end{equation*}
	The definitions of $\mathfrak{C}$, $\mathfrak{L}_0$, and $\varsigma$ imply that
	\begin{equation*}
		\begin{aligned}
			\mathfrak{C} + \frac{\rho_0}{2} - \varsigma
			= {} & \mathfrak{L}_0
			+ \widetilde{C} \rho_0 \norm{\bar{\vd}_0}^2
			+ \frac{2}{9} \rho_0 (6 n \alpha + M \alpha^2 \Gamma B)
			+ \dfrac{(n + 1)\rho_0}{2}
			- \varsigma \\
			= {} & \beta M \alpha^2 \rho_0 \dkh{\norm{\vE_{s, 0}}^2 + \norm{\vE_{d, 0}}^2}
			+ \widetilde{C} \rho_0 \norm{\bar{\vd}_0}^2 \\
			& + \frac{2}{9} \rho_0 (6 n \alpha + M \alpha^2 \Gamma B)
			- 2 \\
			\leq {} & \dkh{ \beta M \alpha^2 \rho_0 + \widetilde{C} \rho_0 } P
			+ \frac{2}{9} \rho_0 (6 n \alpha + M \alpha^2 \Gamma B)
			- 2 \\
			\leq {} & - 1,
		\end{aligned}
	\end{equation*}
	where the last inequality uses the condition $\alpha \leq \alpha_\star$. 
	Consequently, the inclusion $\bar{\vx}_K \in \cQ$ holds since
	\begin{equation*}
		\Phi (\bar{\vx}_K)
		\leq
		\mathfrak{C} + \frac{\rho_0}{2}
		\leq
		\varsigma - 1
		\leq
		\varsigma.
	\end{equation*}
	Combining this inclusion with \eqref{eq:localization_Ex_rough} gives
	\begin{equation*}
		\dist (\vx_{i, K}, \cQ)
		\leq
		\norm{\vx_{i,K} - \bar{\vx}_K}
		\leq
		\norm{\vE_{x, K}}
		\leq
		\frac{1}{2}
		< 1,
	\end{equation*}
	for every $i \in \cN$.
	We obtain $\vx_{i, K} \in \cU$. 
	This proves the claim. 
	In particular, we have $\vx_{i, T - 1} \in \cU$ and $\bar{\vx}_{T - 1} \in \cU$. 
	The convexity of $\cU$ guarantees that $\sum_j w_{ij} \vx_{j, T - 1} \in \cU$. 
	According to the update scheme of $\vx_{i, T}$, we have
	\begin{equation*}
		\dist (\vx_{i, T}, \cU)
		\leq \norm{\vx_{i, T} - \sum_{j \in \cN_i} w_{ij} \vx_{j, T - 1}}
		\leq
		\alpha \rho_{T - 1} \norm{\vd_{i, T}}
		\leq
		\alpha \rho_0 D
		\leq
		\frac{1}{2} < 1,
	\end{equation*}
	where the third inequality follows from \eqref{eq:localization_D_bound} and the fourth inequality holds since $\alpha \leq (1 - \rho_W) / (2 \rho_0 D)$. 
	Hence, $\vx_{i, T} \in \cK$ for every $i \in \cN$. 
	Similarly, combining \eqref{eq:average_updates} with $\|\bar{\vd}_T\| \leq \bar{L}_0 \leq D$ gives
	\begin{equation}
		\label{eq:dist-bar-x}
		\dist (\bar{\vx}_T, \cU)
		\leq
		\norm{\bar{\vx}_T - \bar{\vx}_{T - 1}}
		\leq
		\alpha \rho_{T - 1} \norm{\bar{\vd}_T}
		\leq
		\frac{1}{2}
		\leq 1.
	\end{equation}
	As a result, we have $\bar{\vx}_T \in \cK$. 
	This contradicts the definition of $T$. 
	The proof is completed. 
\end{proof}

\subsection{Global convergence and iteration complexity}

\label{subsec:stationarity_complexity}

Based on the previous preparations, we can establish the global convergence guarantee of Algorithm~\ref{alg:defair_pdc} to a Clarke-stationary point.

\begin{theorem}
	\label{thm:global_convergence}
	Suppose that Assumptions~\ref{asmp:smooth} and~\ref{asmp:graph} hold. 
	Let the algorithmic parameters satisfy \eqref{eq:gc_rho_schedule} with $0 < \alpha \leq \alpha_\star$ and $\tau \geq \tau_\star$. 
	Then the following statements hold. 
	\begin{enumerate}[(i)]
		
		\item The sequence $\{\bar{\vx}_k\}$ has at least one accumulation point. 
		
		\item For every accumulation point $\bar{\vx}^\star$ of $\{\bar{\vx}_k\}$, there exist an increasing sequence of indices $\{k_j\}_{j \geq 1}$ and a vector $\vy^\star \in \Delta_n$ such that $\lim_{j \to \infty} \bar{\vx}_{k_j} = \bar{\vx}^\star$, $\lim_{j \to \infty} \vy_{k_j} = \vy^\star$, and
		\begin{equation}
			\label{eq:weight}
			y^\star_i \dkh{\Phi (\bar{\vx}^\star) - f_i ( \bar{\vx}^\star)} = 0,
		\end{equation}
		for all $i \in \cN$. 
		Here, $y^\star_i$ is the $i$-th entry of $\vy^\star$. 
		
		\item Every accumulation point of $\{\bar{\vx}_k\}$ is a Clarke-stationary point of problem~\eqref{opt:fair-max}. 
	\end{enumerate}
\end{theorem}

\begin{proof}
	According to Lemma~\ref{lem:localization}, the sequence $\{\bar{\vx}_k\}$ is bounded, which has at least one accumulation point. 
	This proves (i). 
	Let $\bar{\vx}^\star$ be an arbitrary accumulation point of $\{\bar{\vx}_k\}$.

	By Lemma~\ref{lem:localization}, $\mathscr{H}_{\cK} (q)$ holds for every $q \ge 0$. 
	Hence, the estimate in Lemma~\ref{lem:gc_finite_horizon} is valid for every $K \geq 1$. 
	For every $\nu \in \R$, $\vx \in \R^m$, and $\rho \leq \rho_0$, we have
	\begin{equation*}
		h_\rho (\nu, \vx)
		\geq
		\min_{\zeta \in \R} h_\rho (\zeta, \vx)
		=
		\Phi_\rho (\vx)
		\geq
		\Phi (\vx) - \frac{\rho}{2}
		\geq
		\Phi^\star - \frac{\rho_0}{2}.
	\end{equation*}
	Thus, the sequence $\{\mathfrak{L}_k\}$ is bounded below. 
	Let
	\begin{equation}
		\label{eq:def-rk}
		R_k
		:= 
		\dkh{
			\norm{\nabla_{\vx} h_{\rho_k} (\bar{\nu}_k, \bar{\vx}_k)}^2
			+ \abs{\nabla_\nu h_{\rho_k} (\bar{\nu}_k, \bar{\vx}_k)}^2
		}^{1/2}.
	\end{equation}
	Upon taking $K\to\infty$ in \eqref{eq:gc_finite_horizon}, we arrive at
	\begin{equation}
		\label{eq:gc_sum_gradx}
		\sum_{k = 0}^{\infty} \rho_k R_k^2 < \infty,
		\quad
		\sum_{k = 0}^{\infty} \rho_k \norm{\bar{\vd}_k}^2 < \infty,
		\quad\mbox{and}\quad
		\sum_{k = 0}^{\infty} \dfrac{\norm{\vE_{\nu, k}}^2 + \norm{\vE_{x, k - 1}}^2}{\rho_k} < \infty.
	\end{equation}
	For all $k \geq 1$, we define
	\begin{equation*}
		Z_k := \dkh{ R_k^2 + \norm{\bar{\vd}_k}^2 + \dfrac{\norm{\vE_{\nu, k}}^2 + \norm{\vE_{x, k - 1}}^2}{\rho_k^2} }^{1/2}.
	\end{equation*}
	Then it follows from the relationship \eqref{eq:gc_sum_gradx} that
	\begin{equation}
		\label{eq:sum-rho-z}
		\sum_{k = 1}^{\infty} \rho_k Z_k^2 < \infty.
	\end{equation}
	We claim that for every $r > 0$, every $\epsilon > 0$, and every integer $N$, there exists $k \geq N$ such that
	\begin{equation}
		\label{eq:barxk-zk}
		\norm{\bar{\vx}_k - \bar{\vx}^\star} < r,
		\quad\mbox{and}\quad
		Z_k < \epsilon.
	\end{equation}
	Suppose, to the contrary, that for some $r$, $\epsilon$, and $N$ one has $Z_k \geq \epsilon$ whenever $k \geq N$ and $\norm{\bar{\vx}_k - \bar{\vx}^\star} < r$. 
	Since $\bar{\vx}^\star$ is an accumulation point, $\{\bar{\vx}_k\}$ enters the ball $\{\vx \in \R^m \mid \norm{\vx - \bar{\vx}^\star} \leq r / 2\}$ infinitely often. 
	
	If there exists an integer $N_0$ such that $\norm{\bar{\vx}_k - \bar{\vx}^\star} < r$ for all $k \geq N_0$, we can obtain that
	\begin{equation*}
		\sum_{k = 1}^{\infty} \rho_k Z_k^2
		\geq \sum_{k = \max \{N_0, N\}}^{\infty} \rho_k Z_k^2
		\geq \epsilon^2 \sum_{k = \max \{N_0, N\}}^{\infty} \rho_k
		= \infty,
	\end{equation*}
	which contradicts \eqref{eq:sum-rho-z}. 
	Otherwise, one can recursively choose disjoint pairs of indices $N \leq p_1 < q_1 < p_2 < q_2 < \dotsb$ such that
	\begin{equation*}
		\norm{\bar{\vx}_{p_j} - \bar{\vx}^\star} \leq r / 2,
		\quad\mbox{and}\quad
		q_j = \min \hkh{k > p_j \mid \norm{\bar{\vx}_k - \bar{\vx}^\star} \geq r}.
	\end{equation*}
	Then it holds that $\norm{\bar{\vx}_k - \bar{\vx}^\star} < r$ for all $p_j \leq k < q_j$ and $\norm{\bar{\vx}_{q_j} - \bar{\vx}_{p_j}} \geq r / 2$. 
	Since $p_1 \geq N$, the contradiction hypothesis gives $Z_k \geq \epsilon$ for $p_j \leq k < q_j$. 
	Similar to the proof of \eqref{eq:dist-bar-x}, we have
	\begin{equation*}
		\norm{\bar{\vx}_{k + 1} - \bar{\vx}_k}
		\leq
		\alpha \rho_k \norm{\bar{\vd}_{k + 1}}
		\leq
		\alpha D \rho_k,
	\end{equation*}
	for all $k \geq 0$. 
	Then it can be readily verified that
	\begin{equation*}
		\norm{\bar{\vx}_{q_j} - \bar{\vx}_{p_j}}
		\leq \sum_{k = p_j}^{q_j - 1} \norm{\bar{\vx}_{k + 1} - \bar{\vx}_k}
		\leq \alpha D \sum_{k = p_j}^{q_j - 1} \rho_k,
	\end{equation*}
	which together with the inequality $\|\bar{\vx}_{q_j} - \bar{\vx}_{p_j}\| \geq \|\bar{\vx}_{q_j} - \bar{\vx}^\star\| - \|\bar{\vx}_{p_j} - \bar{\vx}^\star\| \geq r / 2$ implies that
	\begin{equation*}
		\sum_{k = p_j}^{q_j - 1} \rho_k \geq \frac{r}{2 \alpha D}.
	\end{equation*}
	A straightforward verification reveals that
	\begin{equation*}
		\sum_{k = 1}^{\infty} \rho_k Z_k^2
		\geq \sum_{j = 1}^{\infty} \sum_{k = p_j}^{q_j - 1} \rho_k Z_k^2
		\geq \epsilon^2 \sum_{j = 1}^{\infty} \sum_{k = p_j}^{q_j - 1} \rho_k
		\geq \sum_{j = 1}^{\infty} \frac{\epsilon^2 r}{2 \alpha D}
		= \infty.
	\end{equation*}
	This leads to a contradiction again. 
	The claim is proved.

	Applying \eqref{eq:barxk-zk} successively with $r = \epsilon = 1 / j$ and with $N$ larger than the previously selected index, we obtain an increasing subsequence $\{k_j\}_{j \geq 1}$ satisfying 
	\begin{equation}
		\label{eq:z-kj}
		\lim_{j \to \infty} \bar{\vx}_{k_j} = \bar{\vx}^\star,
		\quad\mbox{and}\quad
		\lim_{j \to \infty} Z_{k_j} = 0.
	\end{equation}
	Since $\vy_k \in [0, 1]^n$, compactness allows us to pass to a further subsequence, without relabeling,
	such that $\lim_{j \to \infty} \vy_{k_j} = \vy^\star \in [0, 1]^n$. 
	According to the definition of $Z_k$, we have $\lim_{j \to \infty} R_{k_j} = 0$, which further implies that
	\begin{equation*}
		\lim_{j \to \infty} \abs{ \nabla_\nu h_{\rho_{k_j}} (\bar{\nu}_{k_j}, \bar{\vx}_{k_j}) } = 0.
	\end{equation*}
	Moreover, it holds that $\eta_{x, k - 1} / \rho_k = \alpha \rho_{k - 1} / \rho_k \leq \alpha \theta$ and $\delta_k \to 0$ as $k \to \infty$. 
	Then it follows from Lemma~\ref{lem:dual_mismatch_stale} that $\lim_{j \to \infty} |\bar{s}_{k_j}| = 0$. 
	The first identity in \eqref{eq:avg_s_identity} gives 
	\begin{equation}
		\label{eq:simplex}
		\lim_{j \to \infty} \abs{ 1 - \mathbf{1}_n\zz \vy_{k_j} } = 0,
	\end{equation}
	which guarantees that $1 - \mathbf{1}_n\zz \vy^\star = 0$ and $\vy^\star \in \Delta_n$.

	The next thing to do in the proof is to show that \eqref{eq:weight} holds. 
	In view of \eqref{eq:average_updates} and \eqref{eq:z-kj}, we have 
	\begin{equation}
		\label{eq:diff-nu}
		\lim_{j \to \infty} \bar{\vx}_{k_j - 1} = \bar{\vx}^\star,
		\quad
		\lim_{j \to \infty} \vx_{i, k_j - 1} = \bar{\vx}^\star,
		\quad\mbox{and}\quad
		\lim_{j \to \infty} \abs{ \nu_{i, k_j} - \nu_{l, k_j} } = 0,
	\end{equation}
	for all $i, l \in \cN$. 
	Assume, for contradiction, that for some $\hat{i} \in \cN$ one has $y^\star_{\hat{i}} > 0$ and $\Phi (\bar{\vx}^\star) > f_{\hat{i}} (\bar{\vx}^\star)$. 
	We choose $l \in \argmax_{i \in \cN} f_i (\bar{\vx}^\star)$ and define $\mu := f_l (\bar{\vx}^\star) - f_{\hat{i}} (\bar{\vx}^\star) > 0$. 
	Then it follows from \eqref{eq:diff-nu} that
	\begin{equation}
		\label{eq:diff-f-nu}
		\fkh{ f_l (\vx_{l, k_j - 1}) - \nu_{l, k_j} } - \fkh{ f_{\hat{i}} (\vx_{\hat{i}, k_j - 1}) - \nu_{\hat{i}, k_j}} \to \mu,
	\end{equation}
	as $j \to \infty$. 
	Since $\lim_{j \to \infty} y_{\hat{i}, k_j} = y_{\hat{i}}^\star > 0$, there exist $c_0 > 0$ and $j_0 \geq 1$ such that $y_{\hat{i}, k_j} \geq c_0$ for all $j \geq j_0$. 
	By \eqref{eq:compact_y}, we have $f_{\hat{i}} (\vx_{\hat{i}, k_j - 1}) - \nu_{\hat{i}, k_j} > 0$ for all $j \geq j_0$. 
	Combining this relationship with \eqref{eq:diff-f-nu} yields that
	\begin{equation*}
		f_l (\vx_{l, k_j - 1}) - \nu_{l, k_j} \geq \dfrac{\mu}{2},
	\end{equation*}
	for all sufficiently large $j$. 
	Since $\rho_{k_j - 1} \to 0$, we can obtain that
	\begin{equation*}
		\dfrac{ f_l (\vx_{l, k_j - 1}) - \nu_{l, k_j} }{\rho_{k_j - 1}} \geq 1,
	\end{equation*}
	as $j \to \infty$. 
	The update scheme \eqref{eq:compact_y} gives $y_{l, k_j} = 1$. 
	Hence, it holds that $\mathbf{1}_n\zz \vy_{k_j} \geq y_{l, k_j} + y_{\hat{i}, k_j} \geq 1 + c_0$ for all sufficiently large $j$, contradicting \eqref{eq:simplex}. 
	This proves (ii). 
	

	Combining the second identity in \eqref{eq:avg_s_identity} with \eqref{eq:z-kj}, we finally arrive at
	\begin{equation*}
		0 
		= \lim_{j \to \infty} n \bar{\vd}_{k_j}
		= \lim_{j \to \infty} \sum_{i =1}^n y_{i, k_j} \nabla f_i (\vx_{i, k_j - 1})
		= \sum_{i =1}^n y_i^\star \nabla f_i (\bar{\vx}^\star).
	\end{equation*}
	The relationship \eqref{eq:weight} guarantees that $y_i^\star = 0$ for all $i \notin \argmax_{j \in \cN} f_j (\bar{\vx}^\star)$. 
	Since $\vy^\star \in \Delta_n$, we have
	\begin{equation*}
		0 \in
		\operatorname{conv} \hkh{ \nabla f_i (\bar{\vx}^\star) \;\middle|\; i \in \argmax_{j \in \cN} f_j (\bar{\vx}^\star) }
		= \partial \Phi (\bar{\vx}^\star),
	\end{equation*}
	which indicates that $\bar{\vx}^\star$ is a Clarke-stationary point. 
	This proves (iii). 
	We complete the whole proof. 
\end{proof}

Theorem~\ref{thm:global_convergence} not only establishes the asymptotic stationarity of the averaged primal iterates, but also characterizes the limiting behavior of the adaptive weights generated by APDGT. 
Specifically, for any accumulation point of the averaged primal sequence, there exists a subsequence along which the corresponding weights converge to a simplex vector whose nonzero entries are supported precisely on the agents attaining the maximum of local objective values. 
Thus, beyond producing stationary primal iterates, APDGT automatically recovers a worst-agent certificate that identifies the local objective functions governing the finite-max geometry at the limiting point. 

The final task is to derive the iteration complexity of APDGT to reach an $\varepsilon$-stationary point defined in Definition~\ref{def:epsilon_stationary_new}, which is accomplished in the following theorem.

\begin{theorem}
	\label{thm:stationarity_complexity}
	Suppose that all the conditions of Theorem~\ref{thm:global_convergence} hold. 
	Then, for any $\varepsilon > 0$, Algorithm~\ref{alg:defair_pdc} generates an $\varepsilon$-stationary point of problem~\eqref{opt:fair-max} within $O (\varepsilon^{-3})$ iterations. 
\end{theorem}

\begin{proof}
	For each integer $K \geq 2$, we define $\cT_K := \{ k \in \N \mid \lceil K / 2 \rceil \leq k \leq K - 1\}$. 
	Recall that the notation $R_k$ is defined in \eqref{eq:def-rk}. 
	By Lemma~\ref{lem:gc_finite_horizon} and the lower boundedness of $\{\mathfrak{L}_k\}$, there exists a constant $C_{\rm it} > 0$, independent of $K$, such that
	\begin{equation*}
		\frac{\alpha}{8 n} \sum_{k = \lceil K / 2 \rceil}^{K - 1} \rho_k R_k^2
		\leq
		C_{\rm it},
	\end{equation*}
	which implies that
	\begin{equation*}
		\min_{k \in \cT_K} R_k^2
		\leq
		\frac{8 n C_{\rm it}}{\alpha \sum_{k = \lceil K / 2 \rceil}^{K - 1} \rho_k}.
	\end{equation*}
	Since $\tau \geq 1$, we have $K - 1 + \tau \leq \tau K$ and $\lfloor K / 2 \rfloor \geq K / 3$ for $K \geq 2$. 
	Then it can be readily verified that
	\begin{equation*}
		\sum_{k=\lceil K / 2 \rceil}^{K - 1} \rho_k
		\geq
		\left\lfloor \frac{K}{2} \right\rfloor \rho_{K - 1}
		\geq
		\frac{\rho_0}{3} K^{2/3}.
	\end{equation*}
	Consequently, we have
	\begin{equation}
		\label{eq:gc_complexity_rate}
		\min_{k \in \cT_K} R_k^2
		\leq
		C_{\rm res} K^{-2/3},
	\end{equation}
	where $C_{\rm res} := 24 n C_{\rm it} / (\alpha \rho_0)$. 
	Let $k^\star \in \cT_K$ attain the minimum in \eqref{eq:gc_complexity_rate}. 
	Since $k^\star \geq K / 2$, it holds that
	\begin{equation*}
		R_{k^\star}^2 \leq C_{\rm res} K^{-2/3},
		\quad\mbox{and}\quad
		\rho_{k^\star}
		= \rho_0 \dkh{ \frac{\tau}{k^\star + \tau} }^{1 / 3}
		\leq (2 \tau)^{1 / 3}\rho_0 K^{-1/3}.
	\end{equation*}
	It follows from Lemma~\ref{lem:smoothed_gradient_in_hull} that
	\begin{equation*}
		\begin{aligned}
			\dist \dkh{ 0, \cG_{\rho_{k^\star}} (\bar{\vx}_{k^\star}) }
			\leq {} &
			\norm{\nabla_{\vx} h_{\rho_{k^\star}} (\bar{\nu}_{k^\star}, \bar{\vx}_{k^\star})}
			+\bar{L}_0 \abs{\nabla_\nu h_{\rho_{k^\star}} (\bar{\nu}_{k^\star}, \bar{\vx}_{k^\star})} \\
			\leq {} &
			\sqrt{1+\bar{L}_0^2} R_{k^\star}
			\leq
			\sqrt{(1+\bar{L}_0^2) C_{\rm res}} K^{-1/3}.
		\end{aligned}
	\end{equation*}
	Now we choose
	\begin{equation*}
		K
		\geq
		\max \hkh{
			2,
			\left\lceil \dkh{ (1 + \bar{L}_0^2) C_{\rm res} }^{3/2} \varepsilon^{-3} \right\rceil,
			\left\lceil 2 \tau \rho_0^3 \varepsilon^{-3} \right\rceil
		}.
	\end{equation*}
	Then it is clear that $\rho_{k^\star}\leq\varepsilon$ and $\dist \dkh{ 0, \cG_{\rho_{k^\star}} (\bar{\vx}_{k^\star}) } \leq \varepsilon$.
	Since $\cG_{\rho_{k^\star}} (\bar{\vx}_{k^\star}) \subseteq \cG_\varepsilon (\bar{\vx}_{k^\star})$, we arrive at
	\begin{equation*}
		\dist \dkh{ 0, \cG_\varepsilon (\bar{\vx}_{k^\star}) } \leq \varepsilon,
	\end{equation*}
	which indicates that $\bar{\vx}_{k^\star}$ is an $\varepsilon$-stationary point. 
	The proof is completed. 
\end{proof}

The $O (\varepsilon^{-3})$ iteration complexity of APDGT obtained above matches classical results for nonconvex nonsmooth optimization \cite{bian2013worst,wang2026adaptive,wang2023smoothing} and for nonconvex minimax optimization \cite{thekumprampil2019efficient,zhao2024primal}. 

\section{Numerical experiments}

\label{sec:numerics}

Preliminary numerical results are presented in this section to provide additional insights into the performance guarantees of the finite-max model \eqref{opt:fair-max} and the corresponding algorithm APDGT. 
All the experiments are implemented in Python and run on a MacBook Pro with an Apple M2 Pro chip and 16 GB of memory.

\subsection{Performance of APDGT}

\label{subsec:toy_example}

We first validate the effectiveness of the proposed algorithm on problem~\eqref{opt:fair-max}. 
For our testing, the local function of agent $i \in \cN$ is generated by
\begin{equation*}
	f_i (\vx)
	= 2 - \cos (\va_i\zz \vx + \gamma_i) - \sin (\vc_i\zz \vx + \xi_i) + \frac{1}{2} \norm{\vx}^2,
\end{equation*}
with $\vx \in \R^m$. 
The entries of $\va_i$ and $\vc_i$, together with $\gamma_i$ and $\xi_i$, are sampled independently from the uniform distribution on $[-1, 1]$. 
It is evident that the local functions constructed in this manner are inherently nonconvex. 
We set $m = 20$ and $n = 100$. 
Moreover, the parameters of Algorithm~\ref{alg:defair_pdc} are chosen as $\rho_k = 8 (k + 1)^{- 1 / 3}$, $\eta_{x, k} = \rho_k / 40$, and $\eta_{\nu, k} = \rho_k / 40$. 
Three distinct network topologies are evaluated in this experiment, including a ring network, an Erdos-Renyi network, and a $10 \times 10$ grid network. 
For the Erdos-Renyi network, each potential edge is independently generated with a probability of $0.05$. 
The initial point is generated at random. 

To assess the behavior of APDGT, we report two performance metrics. 
The first metric is the regularized stationarity residual $R_k$ defined in \eqref{eq:def-rk}. 
Lemma~\ref{lem:smoothed_gradient_in_hull} illustrates that $R_k$ can be interpreted as a quantitative measure of the stationarity violation for problem~\eqref{opt:fair-max}. 
The second metric is the combined consensus error
\begin{equation*}
	E_k
	:=
	\dkh{
		\norm{\vX_k - \mathbf{1}_n \bar{\vx}_k\zz}^2
		+
		\norm{\boldsymbol{\nu}_k - \mathbf{1}_n \bar{\nu}_k}^2
	}^{1 / 2},
\end{equation*}
which characterizes the extent to which consensus is achieved across the network.

Figure~\ref{fig:network} depicts, in its two subplots, the decay of two performance metrics over the course of iterations under the three different network topologies. 
As can be observed, the convergence behavior of the consensus error exhibits a clear dependence on the underlying network topology. 
The better connected the network, the faster the consensus error decays. 
By contrast, once the consensus error has been reduced to a sufficiently small level, the curves of the stationarity residual corresponding to three networks become nearly indistinguishable. 
This phenomenon can be attributed to the fact that $R_k$ is evaluated at the network-average iterate $(\bar{\nu}_k, \bar{\vx}_k)$.

\begin{figure}[ht]
	\centering
	
	\subfloat[regularized stationarity residual]{
		\label{subfig:stationarity}
		\includegraphics[width=0.45\linewidth]{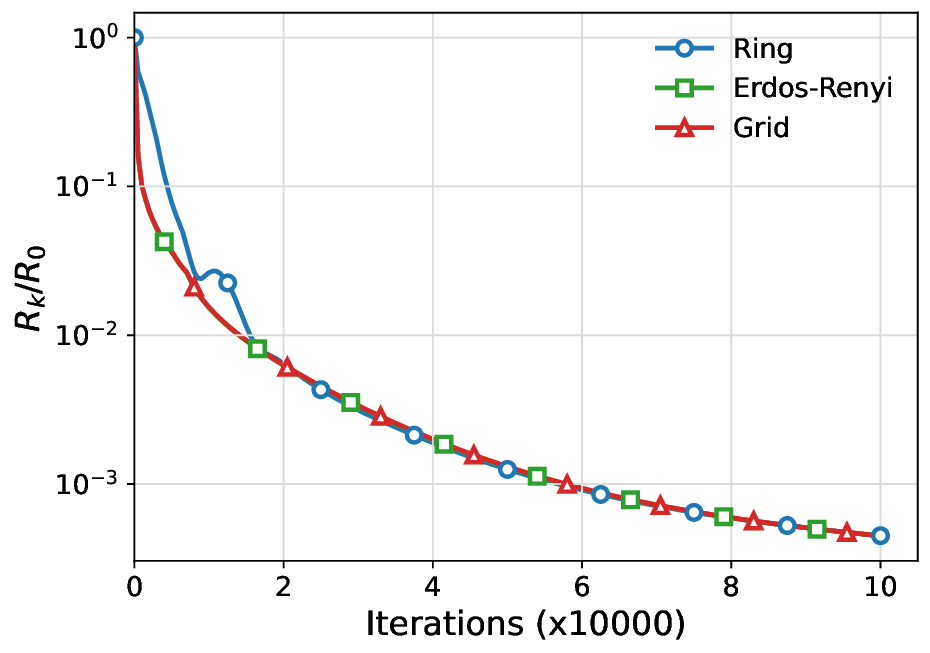}
	}
	\subfloat[combined consensus error]{
		\label{subfig:consensus}
		\includegraphics[width=0.45\linewidth]{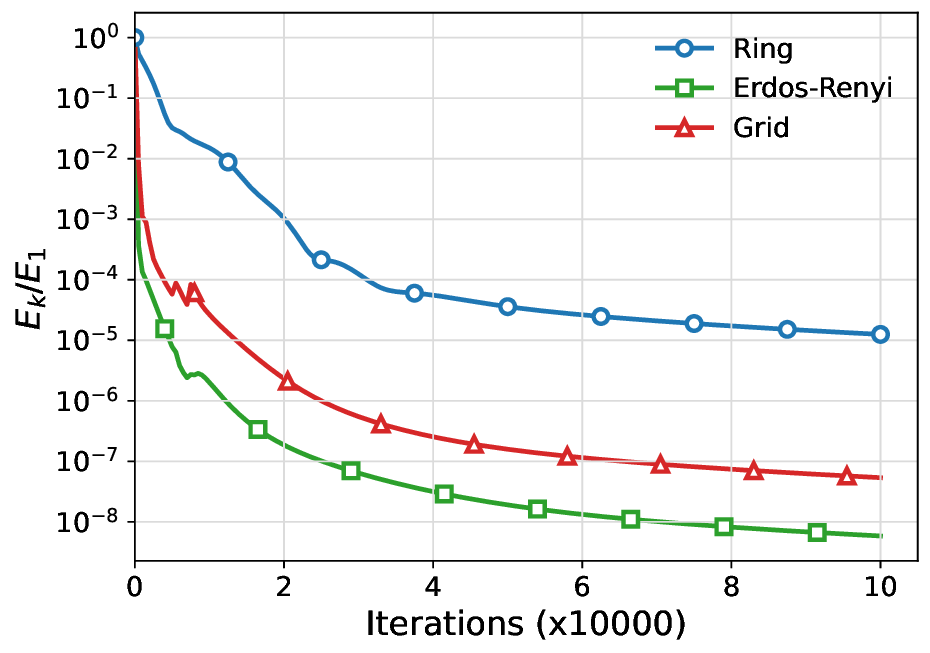}
	}
	
	\caption{Convergence behavior of APDGT in different networks.}
	\label{fig:network}
\end{figure}

We then engage in a comparative evaluation of APDGT and several existing algorithms when local functions are convex. 
Specifically, the local function of agent $i$ is given by
\begin{equation*}
	f_i (\vx)
	= \sqrt{1 + \norm{\vx - \va_i}^2} + \vc_i,
\end{equation*}
with $\vx \in \R^m$. 
The entries of $\va_i$ and $\vc_i$ are independently sampled from the uniform distribution on $[-1, 1]$.
We consider two groups of benchmark algorithms for comparison. 
The first group is designed to solve the penalty model~\eqref{opt:penalty}. 
In particular, we employ the decentralized (sub)gradient descent (DGD) algorithm and the decentralized (sub)gradient tracking (DGT) algorithm, which are referred to as DGD-PM and DGT-PM, respectively. 
The second group applies the DGT algorithm to the $q$FFL model~\eqref{opt:qffl}, and the resulting method is denoted by DGT-$q$FFL. 
We consider two representative values of $q$ to test, including $q = 2$ and $q = 5$.

In this experiment, we set $m = 20$ and conduct all simulations over a ring network of $n = 20$ agents. 
The algorithmic parameters of APDGT are chosen as $\rho_k = 0.015 (k + 1)^{-1/3}$, $\eta_{x, k} = 0.95 \rho_k$, and $\eta_{\nu, k} = 0.95 \rho_k$. 
The penalty parameter in \eqref{opt:penalty} is set to $4.5$. 
We adopt a diminishing stepsize of $(k + 1)^{-0.65}$ for DGD-PM, whereas DGT-PM is implemented with a fixed stepsize of $0.012$. 
For DGT-$q$FFL, the stepsizes are set to $5 \times 10^{-4}$ and $10^{-4}$ for $q = 2$ and $q = 5$, respectively. 
These parameter choices are selected based on empirical tuning and correspond to the best-performing configurations observed in our experiments. 
All algorithms are initialized from the same randomly generated starting point.

Figure~\ref{fig:comm} illustrates how the finite-max objective value $\Phi (\bar{\vx}_k) = \max_{i \in \cN} f_i (\bar{\vx}_k)$ attained by each algorithm decreases with the cumulative communication volume. 
Among all the methods under comparison, APDGT achieves the lowest final objective value. 
For DGT-$q$FFL, increasing the value of $q$ reduces the approximation error inherent in \eqref{opt:qffl}, thereby yielding a lower final objective value. 
Nevertheless, the resulting value remains higher than that attained by APDGT. 
Although DGD-PM and DGT-PM exhibit a relatively rapid decrease during the early stage, their convergence becomes noticeably slower as the iterations proceed. 
Moreover, the objective values to which they eventually converge are also inferior to that of APDGT. 
These observations provide compelling numerical evidence for the advantage of APDGT in solving problem~\eqref{opt:fair-max}.

\begin{figure}[ht]
	\centering
	
	\includegraphics[width=0.45\linewidth]{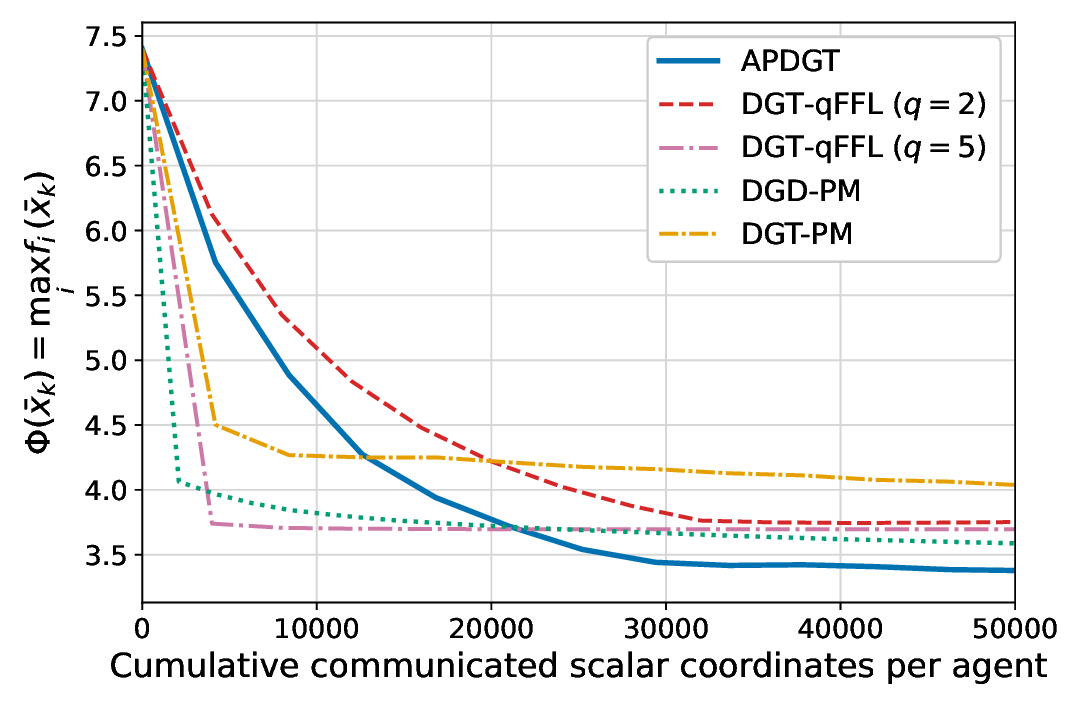}
	
	\caption{Numerical comparison of different algorithms on problem~\eqref{opt:fair-max}.}
	\label{fig:comm}
\end{figure}

\subsection{Performance of finite-max model}

\label{subsec:fashionmnist}

The next series of numerical experiments is designed to investigate the effectiveness of the finite-max model~\eqref{opt:fair-max} in promoting fairness for decentralized optimization. 
Specifically, we consider a logistic regression problem on the Fashion-MNIST dataset \cite{xiao2017fashion}. 
Fashion-MNIST consists of $70000$ samples from $10$ distinct classes, each represented as a $28 \times 28$ grayscale image. 
We vectorize each image into a 784-dimensional feature vector.

The samples of Fashion-MNIST are pooled and then partitioned across $n$ agents. 
We follow the equal-access and single-access split constructions used in fair federated learning experiments \cite{papadaki2022minimax}. 
Three sample-split regimes are considered, including ESG, KSG-3, and SSG. 
Under ESG, every agent receives samples from all $10$ classes with approximately balanced proportions. 
Under SSG, agent $i$ receives samples from the single class indexed by $i \bmod 10$. 
Under KSG-3, agent $i$ receives samples from three classes $\{i \bmod 10, (i + 1) \bmod 10, (i + 2) \bmod 10\}$, and a random permutation of the $10$ classes is applied before assigning samples. 
After this procedure, we split the local dataset into $70\%$ training data $\cD_i^{\rm tr}$ and $30\%$ test data $\cD_i^{\rm te}$ for each agent $i$.


In the decentralized logistic regression problem, we denote $\Theta = [\Theta_0, \dotsc, \Theta_9]\zz \in \R^{10 \times 784}$ and $\boldsymbol{\beta} = (\beta_0, \dotsc, \beta_9)\zz \in \R^{10}$, where $\Theta_j \in \R^{784}$ is the coefficient vector for class $j \in \{0, \dotsc, 9\}$. 
For a vectorized image $u \in \R^{784}$ with class label $v \in \{0, \dotsc, 9\}$, the loss function is
\begin{equation*}
	\ell (\vx; u, v)
	= - \log \dkh{ \frac{\exp (\Theta_v\zz u + \beta_v)}{\sum_{j = 0}^{9} \exp(\Theta_j\zz u + \beta_j)} },
\end{equation*}
where $\vx = (\operatorname{vec} (\Theta)\zz, \boldsymbol{\beta}\zz)\zz$. 
In the training process, the local function of agent $i$ is
\begin{equation*}
	f_i (\vx)
	:=
	\frac{1}{|\cD_i^{\rm tr}|}
	\sum_{(u, v) \in \cD_i^{\rm tr}}
	\ell (\vx; u, v)
	+
	\frac{\mu}{2} \|\vx\|^2,
\end{equation*}
where $\mu = 10^{-4}$ is a ridge regularization parameter and $\|\vx\|^2 = \|\Theta\|^2 + \|\boldsymbol{\beta}\|^2$.

We benchmark the finite-max formulation in \eqref{opt:fair-max} against two alternative modeling paradigms. 
The first is the finite-sum model~\eqref{opt:average}, which is solved using the DGT algorithm and denoted by DGT-average. 
The second is the $q$FFL model~\eqref{opt:qffl}, for which we adopt the DGT-$q$FFL method considered in the preceding experiments. 
As before, we examine two values of $q$, including $q = 2$ and $q = 5$. 
The finite-max model~\eqref{opt:fair-max} is solved by the proposed APDGT algorithm. 
The algorithmic parameters of all competing methods are tuned via grid search, with the best-performing configuration retained for each method. 
All algorithms are run for a common fixed budget of $10000$ iterations. 
This experimental setup ensures that the communication costs incurred by different algorithms are nearly identical.

To evaluate the quality of the solutions delivered by different models, we consider three fairness-aware performance metrics. 
The first metric is the following worst-agent loss on the test data,
\begin{equation}
	\label{eq:fmnist_worst_loss}
	L_{\max}^{\rm te} (\vx)
	:=
	\max_{i \in \cN} \, \ell_i^{\rm te} (\vx),
\end{equation}
where $\ell_i^{\rm te}$ is the test loss on agent $i$ defined by
\begin{equation*}
	\ell_i^{\rm te} (\vx)
	=
	\frac{1}{|\cD_i^{\rm te}|}
	\sum_{(u, v) \in \cD_i^{\rm te}}
	\ell (\vx; u, v).
\end{equation*}
The second metric is the CVaR \cite{rockafellar2002conditional} loss evaluated on the test set as follows,
\begin{equation}
\label{eq:fmnist_cvar_loss}
L_{\rm cvar}^{\rm te} (\vx)
:=
\frac{1}{|\cP_{0.9} (\vx)|}
\sum_{i \in \cP_{0.9} (\vx)}
\ell_i^{\rm te} (\vx),
\end{equation}
where $\cP_{0.9} (\vx)$ indexes the $\lceil 0.1n\rceil$ agents with the largest test losses. 
Thus, $L_{\rm cvar}^{\rm te}$ averages the worst $10\%$ of agents and is less sensitive to a single outlier than \eqref{eq:fmnist_worst_loss}. 
The third metric is the following worst-agent classification error on the test data,
\begin{equation*}
\operatorname{Err}_{\max}^{\rm te} (\vx)
:=
\max_{i \in \cN} \dkh{ 1 - \operatorname{Acc}_i^{\rm te} (\vx) },
\end{equation*}
where the classification accuracy of agent $i$ is given by
\begin{equation*}
\operatorname{Acc}_i^{\rm te} (\vx)
:=
\frac{1}{|\cD_i^{\rm te}|}
\abs{ \hkh{
		(u, v) \in \cD_i^{\rm te}
		\;\middle|\;
		v = \argmax_{j \in \{0, \dotsc, 9\}} \dkh{ \Theta_j\zz u + \beta_j }
} }.
\end{equation*}
For each of the subsequent experiments, we conduct three independent trials.

We begin by comparing the fairness performance of the three models under the most severe data-heterogeneity setting, namely SSG. 
Table~\ref{tab:ssg} reports the corresponding results on a ring network for different numbers of agents, where lower values indicate better performance. 
As evidenced therein, APDGT consistently achieves the best results across all three fairness metrics, demonstrating that the finite-max model~\eqref{opt:fair-max} can substantially enhance fairness among the agents.

\begin{table}[ht]
\caption{Fairness metrics of different algorithms on Fashion-MNIST under SSG.}
\label{tab:ssg}
\centering
\small
\renewcommand{\arraystretch}{1.10}
\begin{tabular*}{\textwidth}{@{\hspace{\tabcolsep}\extracolsep{\fill}}llccc@{\hspace{\tabcolsep}}}
	\hline
	$n$ & Method & CVaR$_{0.9}$ loss & Worst-agent loss & Worst-agent error \\
	\hline
	$20$ & DGT-average & $1.031\pm0.027$ & $1.047\pm0.034$ & $0.413\pm0.008$ \\
	$20$ & DGT-$q$FFL ($q=2$) & $1.004\pm0.043$ & $1.022\pm0.055$ & $0.389\pm0.015$ \\
	$20$ & DGT-$q$FFL ($q=5$) & $0.803\pm0.020$ & $0.813\pm0.020$ & $0.314\pm0.004$ \\
	$20$ & APDGT & $0.765\pm0.023$ & $0.774\pm0.022$ & $0.294\pm0.006$ \\
	\hline
	$50$ & DGT-average & $1.046\pm0.035$ & $1.116\pm0.039$ & $0.435\pm0.011$ \\
	$50$ & DGT-$q$FFL ($q=2$) & $1.032\pm0.045$ & $1.094\pm0.058$ & $0.418\pm0.005$ \\
	$50$ & DGT-$q$FFL ($q=5$) & $0.825\pm0.017$ & $0.877\pm0.028$ & $0.338\pm0.005$ \\
	$50$ & APDGT & $0.788\pm0.017$ & $0.826\pm0.028$ & $0.317\pm0.008$ \\
	\hline
	$100$ & DGT-average & $1.061\pm0.031$ & $1.192\pm0.007$ & $0.463\pm0.003$ \\
	$100$ & DGT-$q$FFL ($q=2$) & $1.039\pm0.032$ & $1.180\pm0.022$ & $0.444\pm0.003$ \\
	$100$ & DGT-$q$FFL ($q=5$) & $0.849\pm0.024$ & $0.942\pm0.034$ & $0.371\pm0.017$ \\
	$100$ & APDGT & $0.814\pm0.013$ & $0.886\pm0.032$ & $0.362\pm0.010$ \\
	\hline
\end{tabular*}
\end{table}

Figure~\ref{fig:ssg} further gives a distributional view for the test losses of different agents when $n = 50$. 
The horizontal axis represents the test-loss bins, while the vertical axis indicates the average number of agents whose test losses fall into each bin. 
DGT-average and DGT-$q$FFL with $q = 2$ place a substantial proportion of agents in relatively low-loss bins. 
Nevertheless, a non-negligible fraction of agents remains in the $0.9$--$1.2$ range. 
Increasing $q$ to 5 alleviates this issue to some extent, but the resulting distribution is still less favorable than that achieved by APDGT. 
By comparison, APDGT concentrates the test losses of all agents within a relatively narrow and moderate-loss range, which completely eliminates the high-loss tail above $0.9$. 
DGT-$q$FFL with $q = 5$ still leaves a small fraction of agents with test losses exceeding $0.9$. 
This more balanced and tightly concentrated loss distribution accounts for the lower CVaR and worst-agent losses attained by APDGT. 

\begin{figure}[ht]
\centering

\subfloat[APDGT]{
	\label{subfig:apdgt}
	\includegraphics[width=0.45\linewidth]{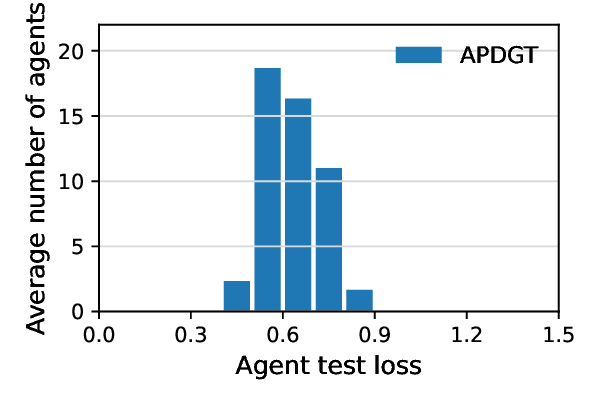}
}
\subfloat[DGT-average]{
	\label{subfig:average}
	\includegraphics[width=0.45\linewidth]{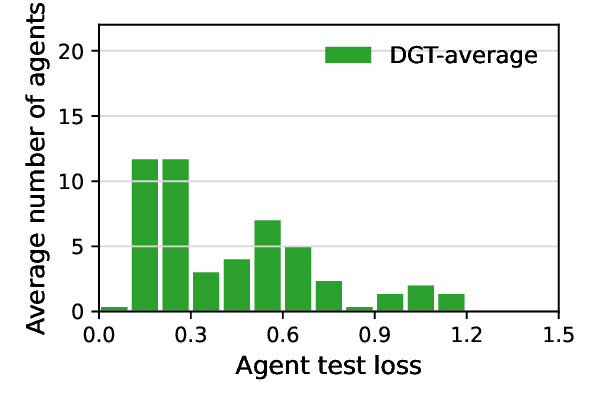}
}

\subfloat[DGT-$q$FFL ($q = 2$)]{
	\label{subfig:qffl2}
	\includegraphics[width=0.45\linewidth]{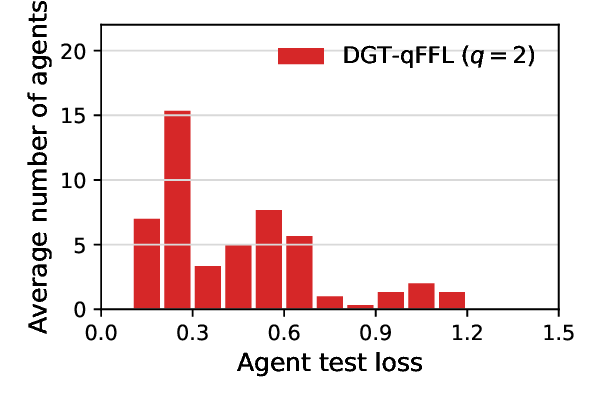}
}
\subfloat[DGT-$q$FFL ($q = 5$)]{
	\label{subfig:qffl5}
	\includegraphics[width=0.45\linewidth]{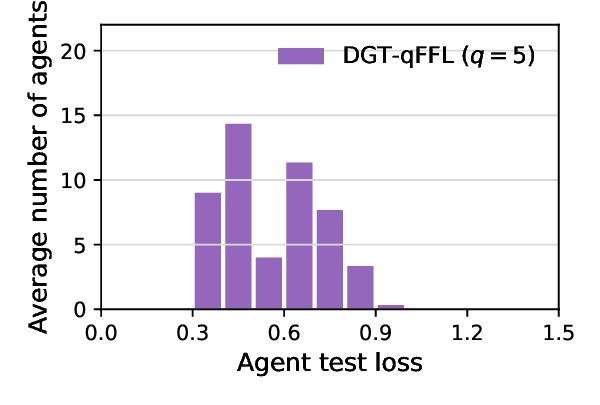}
}

\caption{Distributions of test losses for different agents on Fashion-MNIST under SSG.}
\label{fig:ssg}
\end{figure}

In the final experiment, we investigate the impact of different data distributions on a ring network with $50$ agents. 
Table~\ref{tab:split} reports the fairness metrics achieved by the competing algorithms under three sample-split regimes. 
An interesting observation is that DGT-$q$FFL does not uniformly outperform DGT-average under ESG and KSG-3. 
This behavior can be understood as a consequence of the trade-off between fairness-oriented reweighting and the optimization difficulty induced by the $q$FFL surrogate in \eqref{opt:qffl}. 
By weighting the local gradient of agent $i$ by $[f_i (\vx)]^q$, DGT-$q$FFL emphasizes high-loss agents, but simultaneously induces greater disparity in gradient scales and potentially worsens the conditioning of decentralized optimization. 
Under mildly heterogeneous regimes such as ESG, where local objectives are already well aligned, the fairness gain is limited and may be outweighed by this additional optimization difficulty. 
As data heterogeneity increases, high local losses become more informative of genuine disparities among agents, making the reweighting of DGT-$q$FFL more favorable. 
Nevertheless, APDGT consistently delivers the best performance across all tested settings. 
These results offer additional empirical support for the benefits of the finite-max model~\eqref{opt:fair-max} in promoting fairness across the network.

\begin{table}[ht]
\caption{Fairness metrics on Fashion-MNIST under different sample-split regimes.}
\label{tab:split}
\centering
\small
\renewcommand{\arraystretch}{1.10}
\begin{tabular*}{\textwidth}{@{\hspace{\tabcolsep}\extracolsep{\fill}}llccc@{\hspace{\tabcolsep}}}
	\hline
	Split & Method & CVaR$_{0.9}$ loss & Worst-agent loss & Worst-agent error \\
	\hline
	ESG & DGT-average & $0.509\pm0.021$ & $0.518\pm0.014$ & $0.183\pm0.011$ \\
	ESG & DGT-$q$FFL ($q=2$) & $0.808\pm0.007$ & $0.830\pm0.014$ & $0.296\pm0.003$ \\
	ESG & DGT-$q$FFL ($q=5$) & $0.796\pm0.013$ & $0.806\pm0.014$ & $0.228\pm0.010$ \\
	ESG & APDGT & $0.508\pm0.020$ & $0.518\pm0.014$ & $0.179\pm0.008$ \\
	\hline
	KSG-3 & DGT-average & $0.670\pm0.103$ & $0.717\pm0.110$ & $0.264\pm0.046$ \\
	KSG-3 & DGT-$q$FFL ($q=2$) & $0.826\pm0.282$ & $0.867\pm0.291$ & $0.334\pm0.138$ \\
	KSG-3 & DGT-$q$FFL ($q=5$) & $0.666\pm0.058$ & $0.706\pm0.054$ & $0.255\pm0.021$ \\
	KSG-3 & APDGT & $0.621\pm0.068$ & $0.659\pm0.070$ & $0.237\pm0.030$ \\
	\hline
	SSG & DGT-average & $1.046\pm0.035$ & $1.116\pm0.039$ & $0.435\pm0.011$ \\
	SSG & DGT-$q$FFL ($q=2$) & $1.032\pm0.045$ & $1.094\pm0.058$ & $0.418\pm0.005$ \\
	SSG & DGT-$q$FFL ($q=5$) & $0.825\pm0.017$ & $0.877\pm0.028$ & $0.338\pm0.005$ \\
	SSG & APDGT & $0.788\pm0.017$ & $0.826\pm0.028$ & $0.317\pm0.008$ \\
	\hline
\end{tabular*}
\end{table}

\section{Conclusion}

\label{sec:conclusion}

This paper studied decentralized finite-max optimization problem~\eqref{opt:fair-max}, a formulation that arises naturally when the performance of a network is dictated by its worst-performing component rather than by an aggregate or average criterion.  
The principal obstacle is not nonsmoothness alone, but the global constraint hidden in the selection of the worst local functions. 
The conventional simplex formulation in \eqref{opt:fair-simplex} makes this dependence explicit, yet simultaneously introduces an operation that cannot be carried out by agents possessing only local information. 
By regularizing the inner maximization and exposing its dual structure, we replace this global simplex operation with a common dual variable and locally generated clipped weights. 
The resulting formulation preserves the worst-case character of the model in the vanishing-regularization limit while admitting a genuinely decentralized implementation.

Building upon this structural construction, APDGT combines alternating primal-dual updates with two complementary gradient-tracking mechanisms. 
One tracker coordinates the normalization of the adaptive weights, whereas the other reconstructs the gradient contributed by the near-active agents. 
The convergence analysis reveals how these quantities interact with network disagreement and with the progressively sharper geometry of the regularized dual function. 
A Lyapunov argument establishes the boundedness of iterates and the descent of stationarity residuals. 
The near-active gradient-hull relation transfers these properties back to the original nonsmooth problem. 
Consequently, every accumulation point of the sequence generated by APDGT is Clarke-stationary, and an $\varepsilon$-stationary point can be obtained in $O (\varepsilon^{-3})$ iterations.

In the numerical experiments, the convex benchmark demonstrates that APDGT remains competitive with methods tailored to convex finite-max problems, while retaining applicability to the substantially broader nonconvex regime. 
The practical advantages of APDGT are further validated in a decentralized learning task on real-world datasets, where it significantly outperforms existing techniques. 
These results corroborate the effectiveness of our algorithm in controlling the worst-agent performance under heterogeneous data distributions. 

Several directions merit further investigation. 
Extending our framework to stochastic settings would make it more suitable for large-scale learning, while directed and time-varying networks would broaden its applicability beyond symmetric communication architectures. 

\bibliographystyle{abbrv}
\bibliography{APDGT_library}

\addcontentsline{toc}{section}{References}

\end{document}